\documentclass[a4paper,12pt,draft]{amsart}
\usepackage[margin=30mm]{geometry}
\usepackage{amssymb, amsmath, amsthm, amscd}

\usepackage[T1]{fontenc}
\usepackage{eucal,mathrsfs,dsfont}
\usepackage{color}

\renewcommand{\le}{\leqslant}
\renewcommand{\ge}{\geqslant}

\definecolor{mno}{rgb}{0.5,0.1,0.5}

\newcommand{\R}{\mathds R}

\newcommand{\I}{\mathds 1}

\newtheorem{theorem}{Theorem}[section]
\newtheorem{lemma}[theorem]{Lemma}
\newtheorem{proposition}[theorem]{Proposition}
\newtheorem{corollary}[theorem]{Corollary}

\theoremstyle{definition}

\newtheorem{example}[theorem]{Example}
\newtheorem{remark}[theorem]{Remark}

\begin{document}
\allowdisplaybreaks
\title[Nonlocal Dirichlet Forms With Finite Range Jumps or Large Jumps]
{\bfseries Functional inequalities for Nonlocal Dirichlet Forms With
Finite Range Jumps or Large Jumps}

\author{Xin Chen \quad Jian Wang}
\thanks{\emph{X.\ Chen:}
   Grupo de Fisica Matematica, Universidade de Lisboa, Av Prof Gama Pinto 2, 1649-003 Lisbon, Portugal. \texttt{chenxin\_217@hotmail.com}}

  \thanks{\emph{J.\ Wang:}
   School of Mathematics and Computer Science, Fujian Normal University, 350007 Fuzhou, P.R. China. \texttt{jianwang@fjnu.edu.cn}}

\date{}

\maketitle

\begin{abstract} The paper is a continuation of our paper \cite{WW,
CW}, and it studies functional inequalities for non-local Dirichlet
forms with finite range jumps or large jumps. Let $\alpha\in(0,2)$
and $\mu_V(dx)=C_Ve^{-V(x)}\,dx$ be a probability measure. We present explicit and sharp criteria for the
Poincar\'{e} inequality and the super Poincar\'{e} inequality of the
following non-local Dirichlet form with finite range jump
$$\mathscr{E}_{\alpha, V}(f,f):= \frac{1}{2}\iint_{\{|x-y|\le
1\}}\frac{(f(x)-f(y))^2}{|x-y|^{d+\alpha}}\,dy\, \mu_V(dx);$$ on the
other hand, we give sharp  criteria for the Poincar\'{e}
inequality of the non-local Dirichlet form with large jump as
follows
$$\mathscr{D}_{\alpha, V}(f,f):= \frac{1}{2}\iint_{\{|x-y|>
1\}}\frac{(f(x)-f(y))^2}{|x-y|^{d+\alpha}}\,dy\, \mu_V(dx),$$ and
also derive that the super Poincar\'{e} inequality does not hold for
$\mathscr{D}_{\alpha, V}$. To obtain these results above, some new
approaches and ideas completely different from \cite{WW, CW} are
required, e.g.\ local Poincar\'{e} inequality for
$\mathscr{E}_{\alpha, V}$ and $\mathscr{D}_{\alpha, V}$, and the
Lyapunov condition for $\mathscr{E}_{\alpha, V}$. In particular,
the results about $\mathscr{E}_{\alpha, V}$ show that the
probability measure fulfilling Poincar\'{e} inequality and super
Poincar\'{e} inequality for non-local Dirichlet form with finite
range jump and that for local Dirichlet form enjoy some similar
properties; on the other hand, the assertions for
$\mathscr{D}_{\alpha, V}$ indicate that even if functional
inequalities for non-local Dirichlet form heavily depend on the
density of large jump in the associated L\'{e}vy measure, the
corresponding small jump plays an important role for local super
Poincar\'{e} inequality, which is inevitable to derive super
Poincar\'{e} inequality.

\medskip


\medskip

\end{abstract}
\section{Introduction and Main Results}\label{sec1}
Let $C_b^\infty(\R^d)$ be the set of smooth functions with
bounded derivatives of every order. This paper is concerned with the following two bilinear forms:
\begin{equation*}\mathscr{E}_{\alpha, V}(f,f):=
\frac{1}{2}\iint_{\{|x-y|\le
1\}}\frac{(f(x)-f(y))^2}{|x-y|^{d+\alpha}}\,dy\, \mu_V(dx),\quad
f\in C_b^\infty(\R^d),\end{equation*}
and
\begin{equation*}\mathscr{D}_{\alpha, V}(f,f):=
\frac{1}{2}\iint_{\{|x-y|>
1\}}\frac{(f(x)-f(y))^2}{|x-y|^{d+\alpha}}\,dy\, \mu_V(dx),\quad
f\in C_b^\infty(\R^d),\end{equation*}
where $\alpha\in(0,2)$, $V$ is
a locally bounded Borel measurable function such that
$e^{-V}\in L^1(dx)$, and
\begin{equation*}
\mu_V(dx):=\frac{1}{\int
e^{-V(x)}\,dx}e^{-V(x)}\,dx=:C_Ve^{-V(x)}\,dx
\end{equation*}
 is a probability
measure on $(\R^d, \mathscr{B}(\R^d))$. According to \cite[Theorem
2.1]{CW}, both $(\mathscr{E}_{\alpha, V}, C_b^\infty(\R^d))$ and
$(\mathscr{D}_{\alpha, V}, C_b^\infty(\R^d))$ are
closable bilinear forms on $L^2(\mu_V)$. Therefore, letting
$\mathscr{D}(\mathscr{E}_{\alpha,V})$ and
$\mathscr{D}(\mathscr{D}_{\alpha,V})$  be the closure of
$C_b^\infty(\R^d)$ under the norms $$
\|f\|_{\mathscr{E}_{\alpha,V,1}}:=\big(\|f\|^2_{L^2(\mu_V)}+
\mathscr{E}_{\alpha,V}(f,f)\big)^{1/2}$$
and
$$
\|f\|_{\mathscr{D}_{\alpha,V,1}}:=\big(\|f\|^2_{L^2(\mu_V)}+
\mathscr{D}_{\alpha,V}(f,f)\big)^{1/2}$$ respectively,
$(\mathscr{E}_{\alpha, V}, \mathscr{D}(\mathscr{E}_{\alpha,V}))$ and
$(\mathscr{D}_{\alpha, V}, \mathscr{D}(\mathscr{D}_{\alpha,V}))$ are
regular Dirichlet forms on $L^2(\mu_V)$. The Hunt process
associated with  $(\mathscr{E}_{\alpha, V},
\mathscr{D}(\mathscr{E}_{\alpha,V}))$ is an $\R^d$-valued symmetric
jump process with the finite range jump, while the associated Hunt
process for $(\mathscr{D}_{\alpha, V},
\mathscr{D}(\mathscr{D}_{\alpha,V}))$ is an $\R^d$-valued symmetric
jump process only with the jump larger than 1.

The purpose of this
paper is to study the criteria about Poincar\'{e} inequality and
super Poincar\'{e} inequality for $(\mathscr{E}_{\alpha, V},
\mathscr{D}(\mathscr{E}_{\alpha,V}))$ and $(\mathscr{D}_{\alpha, V},
\mathscr{D}(\mathscr{D}_{\alpha,V}))$.
Recently, functional inequalities have been
established in \cite{WW,CW} for non-local Dirichlet form whose jump kernel has
full support on $\R^d$, i.e.
\begin{equation}\label{dir1}\aligned
D_{\rho,V}(f,f):=&\frac{1}{2}\iint
{\big(f(x)-f(y)\big)^2}\rho(|x-y|)\,dy\,\mu_V(dx),
\endaligned
\end{equation}
where $\rho$ is a strictly positive
measurable function on $\R_+:=(0,\infty)$ such that
\begin{equation*}
\int_{(0,\infty)} \rho(r)\big(1\wedge r^2\big)r^{d-1}\,dr<\infty.
\end{equation*}

\medskip

Comparing with the methods of obtaining Poincar\'{e} type inequalities for $D_{\rho,V}$ in \cite{WW,CW},
in order to get the corresponding functional inequalities for $\mathscr{E}_{\alpha, V}$ and $\mathscr{D}_{\alpha,V}$, there are two  fundamental  differences:
\begin{itemize}
\item [(1)] The efficient approach to yield functional inequalities for $D_{\rho,V}$ is
to check the Lyapunov type condition for the generator associated with $D_{\rho,V}$, which heavily depends on the property of $\rho$. For $D_{\rho,V}$ the Lyapunov function $\phi$ we choose in \cite{WW,CW} is of the form $\phi(x)=|x|^\beta$ with some constant $\beta\in(0,1)$.
   Similar to \cite{CW}, one can apply this test function  $\phi$ into the generator of $\mathscr{D}_{\alpha, V}$, and verify the corresponding
    Lyapunov type condition; however, this test function $\phi$ is not useful for the generator of $\mathscr{E}_{\alpha, V}$.

\item[(2)] Another point on obtaining Poincar\'{e} inequality and super Poincar\'{e} inequality for $D_{\rho,V}$ is to
prove the local Poincar\'{e} inequality and the local super Poincar\'{e} inequality. The local super Poincar\'{e} inequality for $D_{\rho,V}$ is derived by the classical Nash inequality of
     Besov space on $\R^d$ and bounded perturbation of functional inequalities for non-local Dirichlet
form; while the local Poincar\'{e} inequality is easily obtained
for $D_{\rho,V}$  by applying the Cauchy-Swarchz inequality. However we are
unable to use these approaches here, since the jump kernel
is not positive pointwise for both $\mathscr{E}_{\alpha, V}$ and
$\mathscr{D}_{\alpha, V}$.
\end{itemize}

Due to the above differences and difficulties, obtaining the
criteria for Poincar\'{e} inequality and super Poincar\'{e}
inequality for $\mathscr{E}_{\alpha, V}$ and $\mathscr{D}_{\alpha,V}$ requires new approaches and ideas, which include the
following three points:
\begin{itemize}
\item [(1)] The new choice of  Lyapunov function for the generator associated with $\mathscr{E}_{\alpha,V}$, which is efficient to yield the
Lyapunov conditions for  $\mathscr{E}_{\alpha,V}$, and is completely
different from that for $D_{\rho,V}$. (See Lemma \ref{le3.3}.)

\item[(2)] The local Poincar\'{e} inequality for both
$\mathscr{E}_{\alpha,V}$ and $\mathscr{D}_{\alpha,V}$ (see
Propositions \ref{pr2.2} and \ref{pr2.4}), and the local super
Poincar\'{e} inequality for $\mathscr{E}_{\alpha,V}$ (not for
$\mathscr{D}_{\alpha,V}$), where
    we will use some results on the Sobolev embedding theorem in Besov space, e.g.\
    \cite{BK}. (See Proposition \ref{pr2.1}.)

\item[(3)] To show that the super  Poincar\'{e} inequality does not hold for
$\mathscr{D}_{\alpha,V}$ with any locally bounded $V$. (See Section
\ref{section4}.)

\end{itemize}

We are now in a position to state the main results in our paper,
which will be split into the following two parts.

\subsection{Functional Inequalities for $\mathscr{E}_{\alpha,V}$}
For any $r>0$, define
\begin{equation}\label{e0}
k(r):=\inf_{|x|\leqslant r+1}e^{-V(x)},\quad
K(r):=\sup_{|x|\leqslant r}e^{-V(x)}.
\end{equation}

\begin{theorem}\label{th3.1}
$(1)$ Suppose that
\begin{equation}\label{th3.1.1}
\liminf_{|x|\rightarrow \infty}\frac{\inf_{|x|-1 \le |z|\leqslant
|x|-{1}/{2}}e^{-V(z)}} {\sup_{|x| \le |z|\le |x|+1}e^{-V(z)}}
>\frac{1}{\alpha}2^{2d+1}(e+e^{{1}/{2}})(2^{\alpha}-1).\end{equation}
Then the following Poincar\'{e} inequality
\begin{equation}\label{e12}
\begin{split}
\mu_V\big(f^2\big)\leqslant C_1 \mathscr{E}_{\alpha,V}(f,f),\quad f
\in C_b^{\infty}(\R^d),\ \mu_V(f)=0
\end{split}
\end{equation}
holds for some constant $C_1>0$.

$(2)$ If
\begin{equation}\label{e12aa}
\liminf_{|x|\rightarrow \infty}\frac{\inf_{|x|-1 \le |z|\leqslant
|x|-{1}/{2}}e^{-V(z)}} {\sup_{|x| \le |z|\le |x|+1}e^{-V(z)}}
=\infty,
\end{equation}
then there exist constants $C_2$, $C_3>0$ such that the following super
Poincar\'{e} inequality holds
\begin{equation}\label{e13}
\begin{split}
\mu_V(f^2)\leqslant s\mathscr{E}_{\alpha,V}(f,f)+\beta(s)\mu_V(|f|)^2,\quad s>0, f\in C_b^\infty(\R^d),
\end{split}
\end{equation}
where
\begin{equation}\label{e13a}
\begin{split}
\beta(s)=&C_2\Big((1+s^{-{d}/{\alpha}})[\Phi^{-1}(C_3s^{-1})]^{^{d+{d^2}/{\alpha}}}\\
&\qquad\qquad \qquad \times
[K(\Phi^{-1}(C_3s^{-1}))]^{^{1+{d}/{\alpha}}}
[k(\Phi^{-1}(C_3s^{-1}))]^{^{-2-{d}/{\alpha}}}\Big)
\end{split}
\end{equation} and
\begin{equation*}
\Phi(r):=\inf_{|x|\geqslant r}\Big(e^{V(x)}\inf_{|x|-1\le |z|\leqslant |x|-{1}/{2}}e^{-V(z)}\Big).
\end{equation*}
\end{theorem}

\medskip

Though the constant in the right hand side of \eqref{th3.1.1} is far
from optimal, the criteria in Theorem \ref{th3.1} are qualitatively
sharp, which can be seen from the following typical examples. For the proofs of examples, see Section \ref{section3.2}.
\begin{example}\label{Ex1}
\begin{itemize}
\item[(1)] Let
\begin{equation*}
\lambda_0:=2\log\Big[\frac{1}{\alpha}2^{2d+1}(e+e^{{1}/{2}})(2^{\alpha}-1)\Big].
\end{equation*}
 Then, for any probability measure
$\mu_{V_\lambda}(dx)=C_{\lambda} e^{-\lambda |x|}\,dx$ with
$\lambda>\lambda_0$, the Poincar\'{e} inequality (\ref{e12}) holds.

\item[(2)] For probability measure $\mu_{V_\delta}(dx)=C_{\delta} e^{-(1+|x|^\delta)}\,dx$ with $\delta>0$, the super Poincar\'{e} inequality (\ref{e13}) holds
if and only if $\delta>1$, and in this case, it holds with
\begin{equation}\label{ex1.2.1a}
\beta(s)=c_1\exp\Big(c_2\big(1+\log^{\frac{\delta}{(\delta-1)}}(1+{1}/{s})\big)\Big),\quad s>0
\end{equation}
for some positive constants $c_1$ and $c_2$, and equivalently, the Markov semigroup $P_t^{\alpha,V_\delta}$ associated with $\mathscr{E}_{\alpha,V_\delta}$ satisfies
\begin{equation*}
\|P_t^{\alpha,V_\delta}\|_{L^1(\mu_{V_\delta})\to L^\infty(\mu_{V_\delta})}\le \lambda_1\exp\Big(\lambda_2\big(1+\log^{\frac{\delta}{\delta-1}}(1+{1}/{t})\big)\Big),\quad t>0
\end{equation*}
 for some positive constants $\lambda_1$ and $\lambda_2$. Moreover,
\eqref{ex1.2.1a} is sharp in the sense that (\ref{e13}) does not
hold with any rate function $\beta(s)$ such that
\begin{equation}\label{ex1.2.1}
\lim_{s \rightarrow 0} \frac{\log
\beta(s)}{\text{log}^{\frac{\delta}{\delta-1}}(1+s^{-1}) }=0.
\end{equation}

\item[(3)] For probability measure $\mu_{V_\theta}(dx)=C_{\theta}
e^{-|x|\log^\theta(1+|x|)}\,dx$ with $\theta\in\R$, the super
Poincar\'{e} inequality (\ref{e13}) holds if and only if $\theta>0$,
and in this case, it holds with
\begin{equation}\label{ex1.2.2a}
\beta(s)=c_3\exp\Big(1+e^{c_4\log^{\frac{1}{\theta}}(1+{1}/{s})}\Big),\quad s>0
\end{equation}
for some positive constants $c_3$ and $c_4$; moreover,
\eqref{ex1.2.2a} is sharp in the sense that (\ref{e13}) does not
hold with any rate function $\beta(s)$ such that
\begin{equation}\label{ex1.2.2}
\lim_{s \rightarrow 0} \frac{\log\log
\beta(s)}{{\log}^{\frac{1}{\theta}}(1+s^{-1}) }=0.
\end{equation}
{In particular, the Markov semigroup $P_t^{\alpha,V_\theta}$ associated with
$\mathscr{E}_{\alpha,V_\theta}$ is ultracontractive if $\theta>1$, and in this case
\begin{equation*}
\|P_t^{\alpha,V_\theta}\|_{L^1(\mu_{V_\theta})\to L^\infty(\mu_{V_\theta})}\le \lambda_3\exp\Big(1+e^{\lambda_4\log^{\frac{1}{\theta}}(1+{1}/{t})}\Big),\quad t>0
\end{equation*}
holds with some positive constants $\lambda_3$ and $\lambda_4$.}
\end{itemize}
\end{example}

\begin{remark} Example \ref{Ex1} above shows that the property of the probability
measure $\mu_V$ fulfilling Poincar\'{e} inequality and super
Poincar\'{e} inequality for $ \mathscr{E}_{\alpha,V}(f,f)$ is
similar to that for local Dirichlet form
$D^{*}_{V}(f,f):=\frac{1}{2}\int |\nabla f(x)|^2\,\mu_V(dx)$, e.g.\
see \cite[Chapters 1 and 3]{WBook}. On the other hand, Example
\ref{Ex1} also implies that the probability measure $\mu_V$ is
easier to satisfy some functional inequalities for
$\mathscr{E}_{\alpha,V}(f,f)$ than those for $D^{*}_{V}(f,f)$. For
instance, given the probability measure
$\mu_{V_\delta}(dx)=C_{\delta} e^{-(1+|x|^\delta)}\,dx$ with
$\delta>0$, Example \ref{Ex1} (2) indicates that the measure
$\mu_{V_\delta}$ satisfies log-Sobolev inequality for
$\mathscr{E}_{\alpha,V_\delta}(f,f)$ if $\delta>1$; however,
$\mu_{V_\delta}$ satisfies log-Sobolev inequality for
$D^{*}_{V_\delta}(f,f)$ only if $\delta\ge2$, also see
\cite[Chapters 3 and 5]{WBook}.   \end{remark}
\subsection{Functional Inequalities for $\mathscr{D}_{\alpha,V}$}

\begin{theorem}\label{th3.2}
$(1)$ If
\begin{equation}\label{th3.2.1}
\liminf_{|x| \rightarrow \infty}\frac{e^{V(x)}}{|x|^{d+\alpha}}>0,
\end{equation}
then the following weighted Poincar\'{e} inequality
\begin{equation}\label{th3.2.2}
\begin{split}
\int f^2(x)\frac{e^{V(x)}}{1+|x|^{d+\alpha}}\,\mu_V(dx) \leqslant
C_1 \mathscr{D}_{\alpha,V}(f,f),\quad f \in C_b^{\infty}(\R^d),\
\mu_V(f)=0
\end{split}
\end{equation}
holds for some constant $C_1>0$. In particular, the following
Poincar\'{e} inequality
\begin{equation}\label{th3.2.2.1}
\begin{split}
\mu_V(f^2) \leqslant C_2 \mathscr{D}_{\alpha,V}(f,f),\quad f \in
C_b^{\infty}(\R^d),\ \mu_V(f)=0
\end{split}
\end{equation}
holds for some constant $C_2>0$.

$(2)$ For any locally bounded function $V$, the following super
Poincar\'{e} inequality
\begin{equation}\label{th3.2.3}
\mu_V(f^2)\le s\mathscr{D}_{\alpha,V}(f,f)+\beta(s)\mu_V(|f|)^2,\quad
s>0,\ f\in C_b^{\infty}(\R^d)
\end{equation}
does not hold for any rate function $\beta:(0,\infty) \rightarrow
(0,\infty)$.
\end{theorem}

We present the following three remarks on Theorem \ref{th3.2}.

\begin{remark} (1) The condition \eqref{th3.2.1} is sharp for the
Poincar\'{e} inequality \eqref{th3.2.2.1}. For instance, let
$\mu_V(dx):=\mu_\varepsilon(dx)=C_\varepsilon (1+|x|)^{-d-\varepsilon}\,dx$
with $\varepsilon>0$. According to \cite[Corollary 1.2]{WW}, the
following Poincar\'{e} inequality
$$\mu_V(f^2) \leqslant C_3 {D}_{\alpha,V}(f,f): =\frac{C_3 }{2}\iint
\frac{(f(x)-f(y))^2}{|x-y|^{d+\alpha}}\,dy\, \mu_V(dx) $$ holds
for all $f \in C_b^{\infty}(\R^d)$ with $ \mu_V(f)=0$, if and only
if $\varepsilon\ge \alpha$. Note that $\mathscr{D}_{\alpha,V}
(f,f)\le{D}_{\alpha,V}(f,f)$, which along with \eqref{th3.2.1}
indicates that for the probability measure $\mu_\varepsilon$ above,
the Poincar\'{e} inequality \eqref{th3.2.2.1} holds if and only if
$\varepsilon\ge \alpha.$

(2) The weighted function in the weighted Poincar\'{e} inequality
\eqref{th3.2.2} is $$w(x)=\frac{e^{V(x)}}{1+|x|^{d+\alpha}},$$ which
is optimal in the sense that, the inequality \eqref{th3.2.2} fails
if we replace $\omega(x)$ above by a positive function
$\omega^*(x)$, which satisfies that
$$\liminf_{|x|\to\infty}\frac{\omega^*(x)}{\omega(x)}=\infty.$$ The
proof is based on \cite[Theorem 1.4]{CW}
and the fact that $\mathscr{D}_{\alpha,V}
(f,f)\le{D}_{\alpha,V}(f,f)$ for any $f\in C_b^\infty(\R^d)$.

(3) A more important point indicated in Theorem \ref{th3.2} is that
$\mathscr{D}_{\alpha, V}$ satisfies the weighted Poincar\'{e}
inequality \eqref{th3.2.2} (which is stronger than the
Poincar\'{e} inequality \eqref{th3.2.2.1}), but not the super
Poincar\'{e} inequality \eqref{th3.2.3}. The main reason for this
statement is due to the fact that the local super Poincar\'{e}
inequality does not hold for $\mathscr{D}_{\alpha, V}$, while the
local Poincar\'{e} inequality holds. That is, to derive the super Poincar\'{e} inequality for non-local Dirichlet form, we also need some
assumption for the density of small jump for the associated L\'{e}vy measure.
\end{remark}

\bigskip

The remaining part of this paper is organized as follows. In the
next section we present the local super Poincar\'{e} inequality for
$\mathscr{E}_{\alpha,V}$, and the local Poincar\'{e} inequality for
both $\mathscr{E}_{\alpha,V}$ and $\mathscr{D}_{\alpha,V}$, which
yields the weak Poincar\'{e} inequality for $\mathscr{E}_{\alpha,V}$
and $\mathscr{D}_{\alpha,V}$. Section \ref{sec3} is devoted to
functional inequalities for $\mathscr{E}_{\alpha,V}$. We first
derive a new Lyapunov type condition for $\mathscr{E}_{\alpha,V}$,
which along with the results in Section \ref{sec2} enables us to
prove Theorem \ref{th3.1} and also gives us the weighted Poincar\'{e}
inequality for $\mathscr{E}_{\alpha,V}$ (cf.\ Proposition \ref{pro-3.4}). Then, we study the
concentration of measure about the functional
inequalities for $\mathscr{E}_{\alpha,V}$, and present
the proof of Example \ref{Ex1}. To illustrate the differences between
$\mathscr{E}_{\alpha,V}$ and the non-local Dirichelt forms in
\cite{WW,CW}, we also compare these criteria here. In particular, we give a sharp and new example about the Poincar\'{e} inequality
and the log-Sobolev inequality for $D_{\alpha,\delta,V}$, which is defined in \eqref{dir1} by setting $\rho(r)=e^{-\delta r}r^{-(d+\alpha)}$ with $\delta\ge0$ and $\alpha\in (0,2)$. In the last
section, we give the proof of Theorem \ref{th3.2}.

\section{The local Poincar\'{e}-type Inequalities for
$\mathscr{E}_{\alpha,V}$ and $\mathscr{D}_{\alpha,V}$}\label{sec2}

Let $B(x,r)$ be the ball with center $x\in \R^d$ and radius $r>0$.
Let $V$ be a locally bounded measurable function on $\R^d$ such that
$e^{-V}\in L^1(dx)$ and $\mu_V(dx)=C_Ve^{-V(x)}\,dx$ is a
probability measure. For $r>0$, let $K(r)$ and $k(r)$ be the
functions defined by (\ref{e0}).

We begin with the following (classical) local super Poincar\'{e} inequality for Lebesgue measure, which has
been used in the proof of Proposition \ref{pr2.1} below.

\begin{lemma}\label{le2.1}
There exists a constant $C_1>0$ such that the following local
super Poincar\'{e} inequality holds on any ball $B(0,r)$ with
$r>1$:
\begin{equation*}
\begin{split}
 \int_{B(0,r)}& f^2(x) \,dx \\
 &\le  s \iint_{B(0,r+1)\times B(0,r+1)}
 \frac{(f(x)-f(y))^2}{|x-y|^{d+\alpha}}\I_{\{|x-y|\leqslant 1\}}\,dy\,dx\\
 &\quad+C_1  r^{d+{d^2}/{\alpha}}
 \big(1+s^{-d/\alpha}\big)\Big(\int_{B(0,r+1)}|f(x)|\,dx\Big)^2,\qquad s>0,\,\,f\in C_b^\infty(\R^d).
\end{split}
\end{equation*}
\end{lemma}

\begin{proof} For $z\in\R^d$ and $p\ge1$, let $L^p(B(z,1/2),dx)$ be the $L^p$ space with respect to Lebesgue measure for Borel measurable functions defined on
the set $B(z,1/2)$. According to \cite[(2.3)]{BK}, for any
$\alpha\in(0,d \wedge 2)$, there is a constant $c_1>0$ such that for all $z
\in \R^d$ and $f\in C_b^{\infty}(\R^d)$,
\begin{equation*}
\begin{split}
&\|f\|^2_{L^{2d/(d-\alpha)}(B(z,1/2),dx)}\\
&\quad\leqslant c_1 \bigg(\iint_{B(z,1/2)\times B(z,1/2)}
 \frac{(f(x)-f(y))^2}{|x-y|^{d+\alpha}}\,dy\,dx+\|f\|^2_{L^{2}(B(z,1/2),dx)}\bigg).
\end{split}
\end{equation*}
Then, by \cite[Corollary 3.3.4 (2)]{WBook}, also see
\cite[Theorem 4.5 (2)]{W2}, for any $\alpha\in(0,d\wedge 2)$, there
is a constant $c_2>0$ such that for each $z\in \R^d$ and $f\in
C_b^{\infty}(\R^d)$,
\begin{equation}\label{e2}
\begin{split}
\int_{B(z,1/2)} f^2(x) \,dx \leqslant & s \iint_{B(z,1/2)\times
B(z,1/2)}
 \frac{(f(x)-f(y))^2}{|x-y|^{d+\alpha}}\,dy\,dx\\
 &+c_2
 \big(1+s^{-d/\alpha}\big)\Big(\int_{B(z,1/2)}|f(x)|\,dx\Big)^2,\quad s>0.\
\end{split}
\end{equation}

On the other hand, according to \cite[Propositions 3.1 and 3.3]{BK},
for any $\alpha\in[d,2)$ (if $d<2$), there is a constant $c_3>0$ such that for
all $z \in \R^d$ and $f\in C_b^{\infty}(\R^d)$,
\begin{equation*}
\begin{split}
&\|f\|^{2(1+\alpha/d)}_{L^{2}(B(z,1/2),dx)}\\
&\leqslant c_3 \bigg(\iint_{B(z,1/2)\times B(z,1/2)}
 \frac{(f(x)-f(y))^2}{|x-y|^{d+\alpha}}\,dy\,dx+\|f\|^2_{L^{2}(B(z,1/2),dx)}\bigg)\|f\|^{2\alpha/d}_{L^1(B(z,1/2),dx)}.
\end{split}
\end{equation*} By \cite[Corollary 3.3.4 (2)]{WBook} again,
we know that the inequality \eqref{e2} also holds for
$\alpha\in[d,2)$ (possibly with a different constant $c_2>0$). In particular, the constants
$c_1,c_2,c_3$ above do not depend on $z\in \R^d$.

For any $r>1$, we can find a finite set $\Pi_r:=\{z_i\}\subseteq
B(0,r)$ such that
\begin{equation}\label{e3}
B(0,r)\subseteq \bigcup_{z_i\in \Pi_r}B(z_i,1/2),\quad \sharp \,\Pi_r
\leqslant c_4 r^d,
\end{equation}
where $\sharp\,\Pi_r$ denotes the number of the element in the set
$\Pi_r$, and $c_4>0$ is a constant independent of $r$. Therefore, by
(\ref{e2}) (note that according to the argument above it holds for
all $\alpha\in(0,2)$) and (\ref{e3}), we get for each $r>1$ and
$f\in C^\infty_b(\R^d)$.
\begin{align*}
 \int_{B(0,r)} f^2(x)\, dx &\leqslant \sum_{z_i \in \Pi_r}\int_{B(z_i,1/2)}f^2(x)\, dx\\
& \leqslant  \sum_{z_i \in \Pi_r}\bigg[ s \iint_{B(z_i,1/2)\times
B(z_i,1/2)}
 \frac{(f(x)-f(y))^2}{|x-y|^{d+\alpha}}\,dy\,dx\\
&\qquad\qquad+c_2\big(1+s^{-d/\alpha}\big)\Big(\int_{B(z_i,1/2)}|f(x)|\,dx\Big)^2\bigg]\\
&=  \sum_{z_i \in \Pi_r}\bigg[ s \iint_{B(z_i,1/2)\times B(z_i,1/2)}
 \frac{(f(x)-f(y))^2}{|x-y|^{d+\alpha}}\I_{\{|x-y|\leqslant 1\}}\,dy\,dx\\
&\qquad\qquad+c_2\big(1+s^{-d/\alpha}\big)\Big(\int_{B(z_i,1/2)}|f(x)|\,dx\Big)^2\bigg]\\
&\leqslant c_4r^d s \iint_{B(0,r+1)\times B(0,r+1)}
 \frac{(f(x)-f(y))^2}{|x-y|^{d+\alpha}}\I_{\{|x-y|\leqslant 1\}}\,dy\,dx\\
 &\quad+
 c_2c_4r^d\big(1+s^{-d/\alpha}\big)\Big(\int_{B(0,r+1)}|f(x)|\,dx\Big)^2,
\end{align*}
where in the equality above we have used the fact that for every
$x,y \in B(z,1/2)$ and $z\in\R^d$, $|x-y|\leqslant 1$; and the last
inequality follows from the fact that $B(z,1/2)\subseteq B(0,r+1)$
for each $z\in \Pi_r\subset B(0,r)$ and $\sharp \,\Pi_r \leqslant
c_4 r^d$.

The required assertion follows by replacing $c_4r^d s$
with $s$ in the inequality above.
\end{proof}
Now, we turn to the local super Poincar\'{e} inequality for
$\mathscr{E}_{\alpha,V}$.
\begin{proposition}\label{pr2.1}
There is a constant $C_2>0$ such that for each $r>1$, $s>0$ and $f\in
C_b^\infty(\R^d)$,
\begin{equation}\label{pr2.1.1}
\begin{split}
 \int_{B(0,r)} f^2(x) \,\mu_V(dx) \leqslant & s \, \mathscr{E}_{\alpha,V}(f,f)+\beta_r(s)\Big(\int_{B(0,r+1)}|f(x)|\,\mu_V(dx)\Big)^2,
\end{split}
\end{equation}
where $$\beta_r(s)=C_2   \frac{r^{d+{d^2}/{\alpha}}
K(r)^{1+{d}/{\alpha}}}
 {k(r)^{2+{d}/{\alpha}}}\big(1+s^{-d/\alpha}\big).$$
\end{proposition}
\begin{proof}
For any $r>1$, by Lemma \ref{le2.1}, we find that for each
$f\in C_b^\infty(\R^d)$ and $s>0$,
\begin{align*}
\int_{B(0,r)} f^2(x) \,\mu_V(dx) &= C_V\int_{B(0,r)}
f^2(x)e^{-V(x)}\,dx\\
&\leqslant C_VK(r)\int_{B(0,r)}
f^2(x)\,dx\\
&\leqslant C_V K(r) \bigg[s
\iint_{B(0,r+1)\times B(0,r+1)}\frac{(f(x)-f(y))^2}{|x-y|^{d+\alpha}}\I_{\{|x-y|\leqslant 1\}}\,dy\,dx\\
&\qquad \qquad\qquad+C_1  r^{d+{d^2}/{\alpha}} \big(1+s^{-d/\alpha}\big)\Big(\int_{B(0,r+1)}|f(x)|\,dx\Big)^2\bigg]\\
&\leqslant \frac{sK(r)}{k(r)} \iint_{B(0,r+1)\times
B(0,r+1)}\frac{(f(x)-f(y))^2}
{|x-y|^{d+\alpha}}\I_{\{|x-y|\leqslant 1\}}\,dy\,\mu_V(dx)\\
&\qquad+
\frac{C_1r^{d+{d^2}/{\alpha}}K(r)}{C_Vk^2(r)}\big(1+s^{-d/\alpha}\big)\Big(\int_{B(0,r+1)}|f(x)|\,\mu_V(dx)\Big)^2,
\end{align*} where $C_1$ is a positive constant independent of
$r$.

Replacing $s$ with ${s k(r)}/{K(r)}$ in the inequality above and according to the definition of $\beta_r(s)$, we arrive at
\begin{equation*}
\begin{split}
 \int_{B(0,r)} f^2(x) \,\mu_V(dx) \leqslant & s \iint_{B(0,r+1)\times B(0,r+1)}
 \frac{(f(x)-f(y))^2}{|x-y|^{d+\alpha}}\I_{\{|x-y|\leqslant 1\}}\,dy\,\mu_V(dx)\\
 &+\beta_r(s)\Big(\int_{B(0,r+1)}|f(x)|\,\mu_V(dx)\Big)^2,\quad
 s>0,
\end{split}
\end{equation*}
which implies the required assertion.\end{proof}

Next, we will present the local Poincar\'{e} inequality for
$\mathscr{E}_{\alpha,V}$, which is inspired by the proofs of
\cite[Theorem 5.1]{CKK} and \cite[Theorem 2.2]{E}, see also
\cite[Section 1]{HKS}.

\begin{proposition}\label{pr2.2} There is a constant $C_3>0$ such
that for each $r>1$ and $f\in C_b^\infty(\R^d)$,
\begin{equation}\label{pr2.2.1}
\begin{split}
\int_{B(0,r)}& \bigg(f(x)-\frac{\int_{B(0,r)} f(x)\,\mu_V(dx)}{\mu_V(B(0,r))}\bigg)^2\,\mu_V(dx)\\
&\leqslant \frac{C_{3}K(r)r^{3d}}{k(r)}\iint_{B(0,r+1)\times
B(0,r+1)} \frac{(f(x)-f(y))^2}{|x-y|^{d+\alpha}}\I_{\{|x-y|\leqslant
1\}}\,dy\,\mu_V(dx)\\
&\leqslant \frac{C_{3}K(r)r^{3d}}{k(r)}\mathscr{E}_{\alpha,V}(f,f).
\end{split}
\end{equation}
\end{proposition}
\begin{proof}
Let $m(A):=\int_A dx$ be the volume of a Borel set $A\subseteq \R^d$
with respect to Lebesgue measure. For any Borel set $A$ with
$m(A)>0$ and $f\in C_b^{\infty}(\R^d)$, set
\begin{equation*}
f_{A}:=\frac{1}{m(A)}\int_A f(x)\,dx.
\end{equation*}
First, there are two positive constants $c_1$, $c_2$ such that for
any $z\in \R^d$,
\begin{equation}\label{e5}
\begin{split}
&\int_{B(z,1/6)}\big(f(x)-f_{B(z,1/6)}\big)^2\,dx\\
&=\frac{1}{(m(B(z,1/6)))^2}\int_{B(z,1/6)}\Big(\int_{B(z,1/6)}(f(x)-f(y))\,dy\Big)^2\,dx\\
&\leqslant c_1 \int_{B(z,1/6)}\Big(\int_{B(z,1/6)}\frac{(f(x)-f(y))^2}
{|x-y|^{d+\alpha}}\,dy\Big)\Big(\int_{B(z,1/6)}|x-y|^{d+\alpha}\,dy\Big)\,dx\\
&\leqslant c_2 \iint_{B(z,1/6)\times B(z,1/6)}\frac{(f(x)-f(y))^2}
{|x-y|^{d+\alpha}}\I_{\{|x-y|\leqslant 1\}}\,dy\,dx,
\end{split}
\end{equation}
where the first inequality follows from the Cauchy-Schwartz
inequality, and in the second inequality we have used the fact that
$|x-y|\leqslant 1$ for every $x$, $y \in B(z,1/6)$ and $z\in\R^d$.

Second, for any $z_1$, $z_2 \in \R^d$ with $B(z_1,1/6)\bigcap
B(z_2,1/6) \neq \emptyset$, there are two constants $c_3$, $c_4>0$
independent of $z_1$, $z_2\in\R^d$ such that
\begin{equation}\label{e6}
\begin{split}
& (f_{B(z_1,1/6)}-f_{B(z_2,1/6)})^2\\
&=\Big(
\frac{1}{m(B(z_1,1/6))m(B(z_2,1/6))}\int_{B(z_1,1/6)}\int_{B(z_2,1/6)}(f(x)-f(y))\,dy\,dx\Big)^2\\
&\leqslant
c_3\int_{B(z_1,1/6)}\Big(\int_{B(z_2,1/6)}\frac{(f(x)-f(y))^2}{|x-y|^{d+\alpha}}\,dy\Big)
\Big(\int_{B(z_2,1/6)}|x-y|^{d+\alpha}\,dy\Big)\,dx\\
&\leqslant c_4\iint_{B(z_1,1/2)\times B(z_1,1/2)}
\frac{(f(x)-f(y))^2}{|x-y|^{d+\alpha}}\I_{\{|x-y|\leqslant
1\}}\,dy\,dx.
\end{split}
\end{equation}
For the first inequality we have also used the Cauchy-Schwartz
inequality, and the second inequality follows from the fact that
$B(z_1,1/6)\bigcup B(z_2,1/6) \subseteq B(z_1,1/2)$.

As before, for each $r>1$,  we can find a finite set
$\Pi_r:=\{z_i\}\subseteq B(0,r)$ such that
 \begin{equation*}
0 \in \Pi_r,\,\, B(0,r)\subseteq \bigcup_{z_i\in
\Pi_r}B(z_i,1/6),\quad\sharp\, \Pi_r \leqslant c_5 r^d,
\end{equation*}
where $c_5>0$ is a constant independent of $r$.

Next, for a fixed $z\in \Pi_r$, we can find a sequence
$\{z_i\}_{i=1}^n \subseteq \Pi_r$ such that $z_1=z$, $z_n=0$, $z_i
\neq z_j$ if $i \neq j$,  and $B(z_i,1/6)\bigcap B(z_{i+1},1/6) \neq
\emptyset$ for every $1\leqslant i \leqslant n-1$. Hence, there
exist $c_6$, $c_7>0$ independent of $r>0$ and $z\in\Pi_r$ such that
\begin{equation}\label{e7}
\begin{split}
&\int_{B(z,1/6)}\Big(f(x)-f_{B(0,1/6)}\Big)^2\,dx\\
&=\int_{B(z,1/6)}\Big((f(x)-f_{B(z,1/6)})+\sum_{i=1}^{n-1}(f_{B(z_i,1/6)}
-f_{B(z_{i+1},1/6)})\Big)^2\,dx\\
&\leqslant n
\bigg(\int_{B(z,1/6)}\big(f(x)-f_{B(z,1/6)}\big)^2\,dx\\
&\qquad\qquad\qquad+\sum_{i=1}^{n-1}
\int_{B(z,1/6)}\big(f_{B(z_i,1/6)}
-f_{B(z_{i+1},1/6)}\big)^2 \,dx \bigg)\\
&\leqslant c_6r^d \sum_{i=1}^n\iint_{B(z_i,1/2)\times B(z_i,1/2)}
\frac{(f(x)-f(y))^2}{|x-y|^{d+\alpha}}\I_{\{|x-y|\leqslant 1\}}\,dy\,dx\\
&\leqslant c_7r^{2d}\iint_{B(0,r+1)\times B(0,r+1)}
\frac{(f(x)-f(y))^2}{|x-y|^{d+\alpha}}\I_{\{|x-y|\leqslant
1\}}\,dy\,dx,
\end{split}
\end{equation}
where in the second inequality we have used (\ref{e5}), (\ref{e6})
and the fact that $n\leqslant c_5 r^d$, and the last inequality
follows from the facts that $B(z_i,1/2)\subseteq B(0,r+1)$ for any $z_i\in \Pi_r$ and
$n\leqslant c_5 r^d$.

Therefore, by (\ref{e7}), for each $r>1$,
\begin{align*}
& \int_{B(0,r)} \bigg(f(x)-\frac{\int_{B(0,r)} f(x)\mu_V(dx)}{\mu_V(B(0,r))}\bigg)^2\,\mu_V(dx)\\
&\leqslant \int_{B(0,r)} \big(f(x)-f_{B(0,1/6)}\big)^2\,\mu_V(dx)\\
&\leqslant c_8K(r)\int_{B(0,r)} \big(f(x)-f_{B(0,1/6)}\big)^2\,dx\\
&\leqslant c_8K(r)\sum_{z_i\in \Pi_r}\int_{B(z_i,1/6)} \big(f(x)-f_{B(0,1/6)}\big)^2\,dx\\
&\leqslant c_9K(r)r^{3d}\iint_{B(0,r+1)\times B(0,r+1)}
\frac{(f(x)-f(y))^2}{|x-y|^{d+\alpha}}\I_{\{|x-y|\leqslant 1\}}\,dy\,dx\\
&\leqslant \frac{c_{10}K(r)r^{3d}}{k(r)}\iint_{B(0,r+1)\times
B(0,r+1)} \frac{(f(x)-f(y))^2}{|x-y|^{d+\alpha}}\I_{\{|x-y|\leqslant
1\}}\,dy\,\mu_V(dx),
\end{align*}
where $c_8$, $c_9$ and $c_{10}$ are some positive constants independent of
$r$. This completes the proof.
\end{proof}

\medskip

We have derived the local super Poincar\'{e} inequality and the
local Poincar\'{e} inequality for $\mathscr{E}_{\alpha,V}$. In
particular, for local super Poincar\'{e} inequality we have used
the embedding theorem for subsets of $\R^d$ in the Besov space, but
one can not apply such embedding theorem in the context of
$\mathscr{D}_{\alpha,V}$, since the part of the finite range jump in
the associated kernel is removed. We believe that the local super
Poincar\'{e} inequality does not hold for $\mathscr{D}_{\alpha,V}$, see Remark \ref{remf} (2) below.
However, we still can prove the following local Poincar\'{e}
inequality for $\mathscr{D}_{\alpha,V}$.

\begin{proposition}\label{pr2.4}
There exists a constant $C_4>0$, such that for any $r>3$ and $f\in C_b^{\infty}(\R^d)$,
\begin{equation}\label{pr2.4.1}
\begin{split}
&\int_{B(0,r)} \bigg(f(x)-\frac{\int_{B(0,r)} f(x)\,\mu_V(dx)}{\mu_V(B(0,r))}\bigg)^2\,\mu_V(dx)\\
&\leqslant \frac{C_{4}K(r)r^{2d+\alpha}}{k(r)}\iint_{B(0,r+1)\times
B(0,r+1)} \frac{(f(x)-f(y))^2}{|x-y|^{d+\alpha}}\I_{\{|x-y|>
1\}}\,dy\,\mu_V(dx)\\
&\leqslant \frac{C_{4}K(r)r^{2d+\alpha}}{k(r)}\mathscr{D}_{\alpha,V}(f,f).
\end{split}
\end{equation}
\end{proposition}
\begin{proof}
Throughout the proof, all the constants $c_i$  $(i\ge1)$ are positive and independent of $r>0$ and $z\in \R^d$.
As before, for each $r>3$,  we can find a finite set
$\Pi_r:=\{z_i\}\subseteq B(0,r)$ such that
 \begin{equation*}
0 \in \Pi_r,\,\, B(0,r)\subseteq \bigcup_{z_i\in
\Pi_r}B(z_i,1/2),\quad\sharp\, \Pi_r \leqslant c_1 r^d.
\end{equation*}
Next, we split
the set $\Pi_r$ as $\Pi_r=\Pi_r^1 \bigcup \Pi_r^2$, where
\begin{equation*}
\Pi_r^1:=\Big\{z \in \Pi_r:\ \text{dist}\big(B(z,1/2),B(0,1/2)\big)>1\Big\},
\end{equation*}
\begin{equation*}
\Pi_r^2:=\Big\{z \in \Pi_r:\ \text{dist}\big(B(z,1/2),B(0,1/2)\big)\le 1\Big\},
\end{equation*}
and $\text{dist}(A,B)$ denotes the distance between the subsets $A$, $B$ in $\R^d$.

For each $z\in \Pi_r^1$, we have
\begin{equation}\label{e2a}
\begin{split}
& \int_{B(z,1/2)}\big(f(x)-f_{B(0,1/2)}\big)^2 \,dx\\
&=
\frac{1}{(m(B(0,1/2)))^2}\int_{B(z,1/2)}
\Big(\int_{B(0,1/2)} \big(f(x)-f(y)\big)\,dy\Big)^2\,dx \\
&\le c_2 \int_{B(z,1/2)}\Big(\int_{B(0,1/2)}
\frac{(f(x)-f(y))^2}{|x-y|^{d+\alpha}}\,dy\Big)
\Big(\int_{B(0,1/2)} |x-y|^{d+\alpha}\,dy\Big)\,dx\\
&\le c_3 r^{d+\alpha} \iint_{B(0,r+1)\times B(0,r+1)}
\frac{(f(x)-f(y))^2}{|x-y|^{d+\alpha}}\I_{\{|x-y|>1\}}\,dy\,dx,
\end{split}
\end{equation}
Here, the first inequality follows from the Cauchy-Schwartz
inequality, and in the last inequality we have used the facts that for all $z\in \Pi_r^1$, $B(z,1/2)\subset B(0,r+1)$; and if  $z\in \Pi_r^1$, then for each
$x\in B(z,1/2)$ and $y\in B(0,1/2)$,
$1<|x-y|\le 2(r+1)$.

For each $z\in \Pi_r^2$, since $r>3$, there exists $z_0 \in B(0,r)$ such that
for each $x\in B(z_0, 1/2)$ and $y \in B(z,1/2)\bigcup B(0,1/2)$, it holds that $|x-y|>1$. Hence,
\begin{equation*}
\begin{split}
& \int_{B(z,1/2)}\big(f(x)-f_{B(0,1/2)}\big)^2 dx\\
& \le 2 \int_{B(z,1/2)}\big(f(x)-f_{B(z_0,1/2)}\big)^2 dx+
2\int_{B(z,1/2)}\big(f_{B(z_0,1/2)}-f_{B(0,1/2)}\big)^2 dx.
\end{split}
\end{equation*}
Since for $x\in B(z_0, 1/2)$ and $y \in B(z,1/2),$ $1<|x-y|\le
2(r+1),$ we can follow the proof of (\ref{e2a}) and get that
\begin{equation*}
\begin{split}
&\int_{B(z,1/2)}\big(f(x)-f_{B(z_0,1/2)}\big)^2 \,dx\\
& \le c_3 r^{d+\alpha}
\iint_{B(0,r+1)\times B(0,r+1)}
\frac{(f(x)-f(y))^2}{|x-y|^{d+\alpha}}\I_{\{|x-y|>1\}}\,dy\,dx.
\end{split}
\end{equation*}
On the other hand, according to the argument of (\ref{e6}) and noticing that for each $x\in B(z_0, 1/2)$ and $y \in B(0,1/2)$,  $1<|x-y|\le 2(r+1)$, we have
\begin{equation*}
\begin{split}
&\big(f_{B(0,1/2)}-f_{B(z_0,1/2)}\big)^2\\
& \le c_4 r^{d+\alpha}
\iint_{B(0,r+1)\times B(0,r+1)}
\frac{(f(x)-f(y))^2}{|x-y|^{d+\alpha}}\I_{\{|x-y|>1\}}
\,dy\,dx.
\end{split}
\end{equation*}
Combining both estimates above, we obtain that for each $z\in \Pi_r^2$,
\begin{equation}\label{e3a}
\begin{split}
&\int_{B(z,1/2)}\big(f(x)-f_{B(0,1/2)}\big)^2 \,dx\\
&\le c_5 r^{d+\alpha}
\iint_{B(0,r+1)\times B(0,r+1)}
\frac{(f(x)-f(y))^2}{|x-y|^{d+\alpha}}\I_{\{|x-y|>1\}}\,dy\,dx.
\end{split}
\end{equation}

Therefore, by \eqref{e2a} and (\ref{e3a}), for each $r>3$,
\begin{equation*}
\begin{split}
& \int_{B(0,r)} \bigg(f(x)-\frac{\int_{B(0,r)} f(x)\mu_V(dx)}{\mu_V(B(0,r))}\bigg)^2\,\mu_V(dx)\\
&\leqslant \int_{B(0,r)} \big(f(x)-f_{B(0,1/2)}\big)^2\,\mu_V(dx)\\
&\leqslant c_6K(r)\int_{B(0,r)} \big(f(x)-f_{B(0,1/2)}\big)^2\,dx\\
&\leqslant c_6K(r)\sum_{z_i\in \Pi_r}\int_{B(z_i,1/2)} \big(f(x)-f_{B(0,1/2)}\big)^2\,dx\\
&\leqslant c_7K(r)r^{2d+\alpha}\iint_{B(0,r+1)\times B(0,r+1)}
\frac{(f(x)-f(y))^2}{|x-y|^{d+\alpha}}\I_{\{|x-y|> 1\}}\,dy\,dx\\
&\leqslant \frac{c_{8}K(r)r^{2d+\alpha}}{k(r)}\iint_{B(0,r+1)\times
B(0,r+1)} \frac{(f(x)-f(y))^2}{|x-y|^{d+\alpha}}\I_{\{|x-y|>
1\}}\,dy\,\mu_V(dx),
\end{split}
\end{equation*}
which completes the proof.
\end{proof}
\begin{remark}
The constants $r^{3d}$ and $r^{2d+\alpha}$ in the local Poincar\'{e} inequality
\eqref{pr2.2.1} and (\ref{pr2.4.1}) are not optimal, and they come from counting the
number of elements in $\Pi_r$. By taking a cover with some
intersection property, we can expect to get better estimates, e.g.\
see \cite[Lemma 5.11]{CKK}. However, the estimates here are
enough for our application.
\end{remark}

\medskip

As a direct consequence of Propositions \ref{pr2.2} and \ref{pr2.4}, we can derive the following weak
Poincar\'{e} inequality for $\mathscr{E}_{\alpha,V}$ and $\mathscr{D}_{\alpha,V}$, by the local Poincar\'{e} inequality
\eqref{pr2.2.1} and \eqref{pr2.4.1}, respectively.
\begin{proposition}\label{pr2.3} $(1)$ There is a constant $C_5>0$ such that for every $s>0$ and $f\in
C_b^{\infty}(\R^d)$ with $\mu_V(f)=0$,
\begin{equation*}
\begin{split}
\mu_V(f^2)\leqslant C_5\alpha_1(s)
\mathscr{E}_{\alpha,V}(f,f)+s\|f\|_{\infty}^2,
\end{split}
\end{equation*} where
\begin{equation*}
\alpha_1(s):=\inf \bigg\{\frac{r^{3d}K(r)}{k(r)}:\
\mu_V(B(0,r)^c)\leqslant \frac{s}{1+s}\textrm{ and }r>1\bigg\}.
\end{equation*}
$(2)$ There is a constant $C_6>0$ such that for every $s>0$ and $f\in
C_b^{\infty}(\R^d)$ with $\mu_V(f)=0$,
\begin{equation*}
\begin{split}
\mu_V(f^2) \leqslant C_6\alpha_2(s)
\mathscr{D}_{\alpha,V}(f,f)+s\|f\|_{\infty}^2,
\end{split}
\end{equation*} where\begin{equation*}
\alpha_2(s):=\inf \bigg\{\frac{r^{2d+\alpha}K(r)}{k(r)}:\
\mu_V(B(0,r)^c)\leqslant \frac{s}{1+s}\textrm{ and }r>3\bigg\}.
\end{equation*}
\end{proposition}

\begin{proof} The proof is based on \cite[Theorem 4.3.1]{WBook} (see also \cite[Theorem 3.1]{RW}). Here we only prove assertion (1), since the proof of assertion (2) is similar. First, according to \eqref{pr2.2.1}, there exists a constant $c_1>0$ such that for any $r>1$ and $f\in C_b^\infty(\R^d)$,
$$\mu_V(f^2\I_{B(0,r)})\le \frac{c_1 K(r)r^{3d}}{k(r)}\mathscr{E}_{\alpha,V}(f,f)+\frac{\mu_V(f\I_{B(0,r)})^2}{\mu_V(B(0,r))}.$$ For any $s>0$, let $r>1$ such that $\mu_V(B(0,r)^c)\le \frac{s}{1+s}$, i.e.\ $\mu_V(B(0,r))\ge\frac{ 1}{1+s}$. Then, for any $f\in C_b^\infty(\R^d)$ with $\mu_V(f)=0$, one has
$$\mu_V(f\I_{B(0,r)})^2=\mu_V(f\I_{B(0,r)^c})^2\le \frac{s^2}{(1+s)^2}\|f\|_\infty^2.$$ Therefore,
\begin{equation*}
\begin{split}\mu_V(f^2)=&\mu_V(f^2\I_{B(0,r)})+\mu_V(f^2\I_{B(0,r)^c})\\
\le &\frac{c_1K(r)r^{3d}}{k(r)}\mathscr{E}_{\alpha,V}(f,f) +\frac{\mu_V(f\I_{B(0,r)})^2}{\mu(B(0,r))}+\frac{s}{1+s}\|f\|_\infty^2\\
\le &\frac{c_1 K(r)r^{3d}}{k(r)}\mathscr{E}_{\alpha,V}(f,f)+s\|f\|_\infty^2,\end{split}
\end{equation*}which yields the required assertion.
\end{proof}

\section{Functional Inequalities for Dirichlet Forms with Finite Range Jumps}\label{sec3}
\subsection{Lyapunov Type Condition for $\mathscr{E}_{\alpha,V}$}
For any $f$, $g\in C_b^\infty(\R^d)$, let
\begin{equation*}
\mathscr{E}_{\alpha,V}(f,g)=\frac{1}{2}\iint_{\{|x-y|\leqslant 1\}}\frac{(f(x)-f(y))(g(x)-g(y))}
{|x-y|^{d+\alpha}}\,dy\,\mu_V(dx).
\end{equation*}
We define the corresponding truncated Dirichlet form as follows:
\begin{equation*}
\widetilde{\mathscr{E}}_{\alpha,V}(f,g):=\frac{1}{2}\iint_{\{{1}/{2}\leqslant |x-y|\leqslant 1\}}\frac{(f(x)-f(y))(g(x)-g(y))}
{|x-y|^{d+\alpha}}\,dy\,\mu_V(dx).
\end{equation*}

Let $C_c^\infty(\R^d)$ be the set of smooth functions on $\R^d$ with
compact supports. The following result presents the explicit
expression for the generator associated with the truncated Dirichlet
form $\widetilde{\mathscr{E}}_{\alpha,V}$ on $C_c^\infty(\R^d)$.
\begin{lemma}\label{le3.1}
For each $f,g \in C_c^{\infty}(\R^d)$,
\begin{equation*}
\widetilde{\mathscr{E}}_{\alpha,V}(f,g)=-\int f(x)\widetilde{L}_{\alpha,V}g(x)\,\mu_V(dx)
=-\int g(x)\widetilde{L}_{\alpha,V}f(x)\,\mu_V(dx),
\end{equation*}
where
\begin{equation}\label{e8}
\widetilde{L}_{\alpha,V}f(x):=\frac{1}{2}\int_{\{{1}/{2}\leqslant |x-y|\leqslant 1\}}
\big(f(y)-f(x)\big)\frac{\big(1+e^{V(x)-V(y)}\big)}{|x-y|^{d+\alpha}}\,dy.
\end{equation}
\end{lemma}
\begin{proof}
According to \cite[Theorem 2.1]{CW}, for each $f$, $g \in C_c^{\infty}(\R^d)$,
$$\widetilde{\mathscr{E}}_{\alpha,V}(f,g)=-\int f(x)\widetilde{L}^*_{\alpha,V}g(x)\,\mu_V(dx)
=-\int g(x)\widetilde{L}^*_{\alpha,V}f(x)\,\mu_V(dx),$$
where
\begin{equation*}
\begin{split}
\widetilde{L}^*_{\alpha,V}f(x):=&\frac{1}{2}\int_{\{{1}/{2}\leqslant |z|\leqslant 1\}}\!\!
\Big(f(x+z)-f(x)-\nabla f(x)\cdot z\I_{\{|z|\le 1\}}\Big)\frac{\big(1+e^{V(x)-V(x+z)}\big)}{|z|^{d+\alpha}}\,dz\\
&+\frac{1}{4}\nabla f(x)\cdot \bigg[\int_{\{{1}/{2}\leqslant |z| \leqslant 1\}} z
\big(e^{V(x)-V(x+z)}-e^{V(x)-V(x-z)}\big)\frac{1}{|z|^{d+\alpha}}\,dz\bigg]\\
=&\frac{1}{2}\int_{\{{1}/{2}\leqslant |x-y|\leqslant 1\}}
\big(f(y)-f(x)\big)\frac{\big(1+e^{V(x)-V(y)}\big)}{|x-y|^{d+\alpha}}\,dy\\
&-\frac{1}{4}\nabla f(x)\cdot \bigg[\int_{\{{1}/{2}\leqslant |z| \leqslant 1\}} z
\big(e^{V(x)-V(x+z)}+e^{V(x)-V(x-z)}\big)\frac{1}{|z|^{d+\alpha}}\,dz\bigg]\\
=&:\widetilde{L}_{\alpha,V,1}f(x)+\widetilde{L}_{\alpha,V,2}f(x).
\end{split}
\end{equation*}
It is easy to see that for any $f\in C_c^\infty(\R^d)$, $\widetilde{L}_{\alpha,V,1}f(x)$ and $\widetilde{L}_{\alpha,V,2}f(x)$ are well defined.
Changing variable from $z$ to $-z$, we can see that for all $x\in\R^d$,
$\widetilde{L}_{\alpha,V,2}f(x)=0$, which gives us the desired expression (\ref{e8}).
\end{proof}

According to (\ref{e8}), for every $f\in C(\R^d)$ (the set
of continuous functions on $\R^d$) and $x\in\R^d$,
$\widetilde{L}_{\alpha,V}f(x)$ is well defined, and the function
$x\mapsto\widetilde{L}_{\alpha,V}f(x)$ is locally bounded. Then,
repeating the proof of \cite[Propsoition 3.2]{CW}, we get
\begin{lemma}\label{le3.2}
For every $f\in C_c^{\infty}(\R^d)$ and $\phi\in C(\R^d)$ with $\phi>0$,
\begin{equation*}\label{le3.2.1}
-\int f^2\frac{\widetilde{L}_{\alpha,V}\phi}{\phi}d\mu_V\leqslant \widetilde{\mathscr{E}}_{\alpha,V}(f,f).
\end{equation*}
\end{lemma}
Now we present the Lyapunov type condition for
$\widetilde{L}_{\alpha,V}$,
\begin{lemma}\label{le3.3}
Let $\phi \in C(\R^d)$
such that $\phi>1$ and $\phi(x)=e^{|x|}$ for $|x|>1$. If
\begin{equation}\label{e9}
\liminf_{|x|\rightarrow \infty}\frac{\inf_{|x|-1 \le |z|\leqslant
|x|-{1}/{2}}e^{-V(z)}} {\sup_{|x| \le |z|\le |x|+1}e^{-V(z)}}>
\frac{1}{\alpha}2^{2d+1}(e+e^{{1}/{2}})(2^{\alpha}-1),
\end{equation}
then there are positive constants $C_1$, $b$ and $r_0>0$ such that
for all $x\in\R^d$,
\begin{equation}\label{e10}
\widetilde{L}_{\alpha,V}\phi(x)\leqslant -C_1\Big(e^{V(x)}\inf_{|x|-1 \le |z|\leqslant |x|-{1}/{2}}e^{-V(z)}\Big)\phi(x)
+b\I_{B(0,r_0)}(x)
\end{equation}
\end{lemma}
\begin{proof}
It is easy to check that $\widetilde{L}_{\alpha,V}\phi$ is locally bounded. Thus, it suffices to prove (\ref{e10}) for
$|x|$ large enough. First, for $x\in\R^d$ with $|x|\ge 2$,
\begin{equation*}
\begin{split}
\int_{\{{1}/{2}\leqslant |z| \leqslant 1\}}&\big(\phi(x+z)-\phi(x)\big)\frac{1}{|z|^{d+\alpha}}\,dz\\
&=\int_{\{{1}/{2}\leqslant |z| \leqslant 1\}}\big(e^{|x+z|}-e^{|x|}\big)\frac{1}{|z|^{d+\alpha}}\,dz\\
&\leqslant e^{|x|}\int_{\{{1}/{2}\leqslant |z| \leqslant 1\}}
\big(e^{|z|}-1\big)\frac{1}{|z|^{d+\alpha}}\,dz\\
&= c_1e^{|x|}\\
&\leqslant c_1\Big(e^{V(x)}\sup_{|x| \le |z|\le |x|+1}e^{-V(z)}\Big)\phi(x),
\end{split}
\end{equation*}where $$c_1:=\int_{\{{1}/{2}\leqslant |z| \leqslant 1\}}
\big(e^{|z|}-1\big)\frac{1}{|z|^{d+\alpha}}\,dz.$$

Second, for $x\in\R^d$ with $|x|\ge 2$,
\begin{align*}
&\int_{\{{1}/{2}\leqslant |z| \leqslant 1\}}\big(\phi(x+z)-\phi(x)\big)\frac{e^{(V(x)-V(x+z))}}{|z|^{d+\alpha}}\,dz\\
&=e^{V(x)}\bigg(\int_{\{{1}/{2}\leqslant |z| \leqslant 1\}}\big(e^{|x+z|}-e^{|x|}\big)
\frac{e^{-V(x+z)}}{|z|^{d+\alpha}}\,dz\bigg)\\
&\leqslant e^{V(x)}\bigg(\int_{\{{1}/{2}\leqslant |z| \leqslant 1,|x+z|-|x|\leqslant -{1}/{2}\}}
\big(e^{|x+z|}-e^{|x|}\big)
\frac{e^{-V(x+z)}}{|z|^{d+\alpha}}\,dz\\
&\qquad\qquad+\int_{\{{1}/{2}\leqslant |z| \leqslant 1,|x+z|-|x|\geqslant 0\}}
\big(e^{|x+z|}-e^{|x|}\big)
\frac{e^{-V(x+z)}}{|z|^{d+\alpha}}\,dz
\bigg)\\
&\leqslant e^{V(x)}\bigg(\int_{\{{1}/{2}\leqslant |z| \leqslant 1,|x+z|-|x|\leqslant -{1}/{2}\}}
\big(e^{|x|-1/2}-e^{|x|}\big)
\frac{e^{-V(x+z)}}{|z|^{d+\alpha}}\,dz\\
&\qquad\qquad+\int_{\{{1}/{2}\leqslant |z| \leqslant 1,|x+z|-|x|\geqslant 0\}}
\big(e^{|x|+|z|}-e^{|x|}\big)
\frac{e^{-V(x+z)}}{|z|^{d+\alpha}}\,dz
\bigg)\\
&=e^{V(x)}e^{|x|}
\bigg(-\int_{\{{1}/{2}\leqslant |z| \leqslant 1,|x+z|-|x|\leqslant -{1}/{2}\}}
(1-e^{-{1}/{2}})
\frac{e^{-V(x+z)}}{|z|^{d+\alpha}}\,dz\\
&\qquad\qquad\qquad+\int_{\{{1}/{2}\leqslant |z| \leqslant 1,|x+z|-|x|\geqslant 0\}}(e^{|z|}-1)
\frac{e^{-V(x+z)}}{|z|^{d+\alpha}}\,dz
\bigg)\\
&\leqslant e^{V(x)}e^{|x|}\Bigg[-(1-e^{-{1}/{2}})
\Big(\inf_{|x|-1 \le |z|\leqslant |x|-{1}/{2}}e^{-V(z)}\Big)
\int_{\{{1}/{2}\leqslant |z| \leqslant 1,|x+z|-|x|\leqslant -{1}/{2}\}}\frac{1}{|z|^{d+\alpha}}\,dz\\
&\qquad\qquad\qquad+\Big(\sup_{|x|\le |z|\leqslant |x|+1}e^{-V(z)}\Big)
\int_{\{{1}/{2}\leqslant |z| \leqslant
1\}}(e^{|z|}-1)\frac{1}{|z|^{d+\alpha}}\,dz\Bigg],
\end{align*}
where in the first inequality we have removed the subset
$\{z\in\R^d: 1/2\le |z|\le 1,-1/2< |x+z|-|x|< 0\}$ in the domain of
integral, since the integrand is negative on this subset. For
$x\in\R^d$ with $|x|\ge 2$, let $z_0=-{3x}/({4|x|})$. Then,
$$|z_0|=\frac{3}{4}\quad\textrm{ and
}\,\,|x+z_0|-|x|=-\frac{3}{4}.$$  Hence, for every $z \in
B(z_0,\frac{1}{4})$,
\begin{equation*}
\begin{split}
& |z|\geqslant |z_0|-\frac{1}{4}\geqslant \frac{1}{2} , \quad |z|\leqslant |z_0|+\frac{1}{4}=1,\\
& |x+z|-|x|\leqslant (|x+z_0|-|x|)+(|x+z|-|x+z_0|)\leqslant
-\frac{3}{4}+|z-z_0|\leqslant -\frac{1}{2},
\end{split}
\end{equation*}
which implies that
\begin{equation*}
B\Big(z_0,\frac{1}{4}\Big)\subseteq \Big\{z \in \R^d:\
\frac{1}{2}\leqslant |z| \leqslant 1,|x+z|-|x|\leqslant
-\frac{1}{2}\Big\}.
\end{equation*}
According to both conclusions above, we get that
\begin{equation*}
\begin{split}
&\int_{\{{1}/{2}\leqslant |z| \leqslant 1\}}\big(\phi(x+z)-\phi(x)\big)\frac{e^{(V(x)-V(x+z))}}{|z|^{d+\alpha}}dz\\
& \leqslant e^{V(x)}e^{|x|}\bigg[-(1-e^{-{1}/{2}})
\Big(\inf_{|x|-1 \le |z|\leqslant |x|-{1}/{2}}e^{-V(z)}\Big)
m\Big(B\Big(z_0,\frac{1}{4}\Big)\Big)\\
&\qquad\qquad\qquad+\Big(\sup_{|x| \le |z|\le |x|+1}e^{-V(z)}\Big)
\int_{\{{1}/{2}\leqslant |z| \leqslant 1\}}(e^{|z|}-1)\frac{1}{|z|^{d+\alpha}}\,dz\bigg]\\
&\leqslant -c_2\phi(x)e^{V(x)}\Big(\inf_{|x|-1 \le |z|\leqslant |x|-{1}/{2}}e^{-V(z)}\Big)
+c_1\phi(x)e^{V(x)}\Big(\sup_{|x|\le |z|\le |x|+1}e^{-V(z)}\Big),
\end{split}
\end{equation*} where $m(A)$ is Lebesgue measure for the Borel measurable set $A$, and
$$c_2:=(1-e^{-{1}/{2}})m\Big(B\Big(0,\frac{1}{4}\Big)\Big).$$

Combining both estimates above with \eqref{e8}, we know that for
any $x\in\R^d$ with $|x|\ge2$, it holds that
$$\widetilde{L}_{\alpha,V}\phi(x)\le \frac{1}{2}\bigg[-c_2\Big(\inf_{|x|-1 \le |z|\leqslant |x|-{1}/{2}}e^{-V(z)}\Big)
+2c_1\Big(\sup_{|x| \le |z|\le |x|+1}e^{-V(z)}\Big)\bigg]\phi(x)e^{V(x)}.$$

Therefore, if $$\liminf_{|x|\rightarrow \infty}\frac{\inf_{|x|-1 \le |z|\leqslant |x|-{1}/{2}}e^{-V(z)}}
{\sup_{|x| \le |z|\le |x|+1}e^{-V(z)}}>\frac{2c_1}{c_2},$$ then for $|x|$ large enough,
\begin{equation*}
\widetilde{L}_{\alpha,V}\phi(x)\le -C_1\phi(x)e^{V(x)}\Big(\inf_{|x|-1 \le |z|\leqslant |x|-{1}/{2}}e^{-V(z)}\Big)
\end{equation*} holds with some constant $C_1>0$. The required assertion follows from the fact that
$$\frac{2c_1}{c_2}=\frac{2^{2d+1}d}{1-e^{-1/2}}\int_{1/2}^1(e^r-1)r^{-1-\alpha}\,dr<
\frac{1}{\alpha}2^{2d+1}(e+e^{{1}/{2}})(2^{\alpha}-1)$$ and \eqref{e9}.
\end{proof}

\bigskip

Now we present the proof of Theorem \ref{th3.1}.
\begin{proof}[Proof of Theorem \ref{th3.1}]
The proof is the same as that of \cite[Theorem 2.10]{CGWW} and
\cite[Theorem 3.6]{CW} (see also \cite[Theorem 1.1]{WW}), and it is based on Lemma \ref{le3.3} and
the local (super) Poincar\'{e} inequality for
$\mathscr{E}_{\alpha,V}$. Here, we only show the super Poincar\'{e}
inequality (\ref{e13}).  Based on the local Poincar\'{e} inequality in
Proposition \ref{pr2.2},
the proof for the Poincar\'{e} inequality
(\ref{e12}) is similar and even simpler.

According to Lemma \ref{le3.3}, there are constants
$c_1$, $c_2$ and $r_0>1$ such that
\begin{equation*}
\widetilde{L}_{\alpha,V}\phi(x)\leqslant -c_1
\phi(x)e^{V(x)}\Big({\inf_{|x|-1 \le |z|\leqslant
|x|-{1}/{2}}e^{-V(z)}}\Big)+c_2\I_{B(0,r_0)}(x),
\end{equation*}
where $\phi(x)$ is the function given in Lemma \ref{le3.3}.

For
$r>0$, set $$\Phi(r)=\inf_{|x|\ge
r}\Big[e^{V(x)}\Big({\inf_{|x|-1 \le |z|\leqslant
|x|-{1}/{2}}e^{-V(z)}}\Big)\Big].$$ By Lemma \ref{le3.2}, for any
$f \in C_c^{\infty}(\R^d)$ and $r\ge r_0$,
\begin{equation}\label{e14ooo}
\begin{split}
\int_{B(0,r)^c} f^2(x)\,\mu_V(dx)&\leqslant \frac{1}{\Phi(r)}\int f^2(x)\Phi(|x|)\,\mu_V(dx)\\
&\leqslant \frac{1}{\Phi(r)}\int f^2(x)e^{V(x)}\Big({\inf_{|x|-1 \le |z|\leqslant
|x|-{1}/{2}}e^{-V(z)}}\Big)\,\mu_V(dx)\\
&\leqslant -\frac{1}{c_1\Phi(r)}\int \frac{\widetilde{L}_{\alpha,V}\phi(x)}{\phi(x)}f^2(x)\,\mu_V(dx)\\
&\qquad+\frac{c_2}{c_1\Phi(r)}\int_{B(0,r_0)} \frac{f^2(x)}{\phi(x)}\,\mu_V(dx)\\
&\leqslant
\frac{c_3}{\Phi(r)}\bigg[\widetilde{\mathscr{E}}_{\alpha,V}(f,f)+
\int_{B(0,r_0)} f^2(x)\mu_V(dx)\bigg]\\
&\leqslant
\frac{c_3}{\Phi(r)}\bigg[\widetilde{{\mathscr{E}}}_{\alpha,V}(f,f)+
\int_{B(0,r)} f^2(x)\mu_V(dx)\bigg],
\end{split}
\end{equation}
where in the forth inequality we have used the fact that $\phi>1$.

For every $f\in C_b^{\infty}(\R^d)$, there is a sequence of
functions $\{f_n\}_{n=1}^{\infty}\subseteq$ $C_c^{\infty}(\R^d)$
such that
\begin{equation*}
\lim_{n\rightarrow \infty}f_n(x)=f(x),\ \
\sup_n\|f_n\|_{\infty}<\infty,\ \ \sup_n\|\nabla
f_n\|_{\infty}<\infty.
\end{equation*}
Thus, by the dominated convergence
theorem, we get
\begin{equation*}
\lim_{n\rightarrow
\infty}\widetilde{{\mathscr{E}}}_{\alpha,V}(f_n,f_n)=
\widetilde{{\mathscr{E}}}_{\alpha,V}(f,f),
\end{equation*}
$$
\lim_{n\to\infty}\int_{B(0,r)^c} f_n^2(x)\,\mu_V(dx)
=\int_{B(0,r)^c} f^2(x)\,\mu_V(dx),$$ and $$
\lim_{n\to\infty}\int_{B(0,r)} f_n^2(x)\,\mu_V(dx) =\int_{B(0,r)}
f^2(x)\,\mu_V(dx).
$$
Since (\ref{e14ooo}) holds for each $f_n$, letting $n$ tend to
infinity and using the estimates above, we show that
(\ref{e14ooo}) holds for $f\in C_b^{\infty}(\R^d)$.

Hence, for every $r\ge r_0$ and $f \in C_b^{\infty}(\R^d)$,
\begin{equation*}
\begin{split}
\int f^2(x)\,\mu_V(dx)&=\int_{B(0,r)} f^2(x)\,\mu_V(dx)+\int_{B(0,r)^c} f^2(x)\,\mu_V(dx)\\
&\leqslant \frac{c_3}{\Phi(r)}\mathscr{E}_{\alpha,V}(f,f)+
\Big(1+\frac{c_3}{\Phi(r)}\Big)\int_{B(0,r)} f^2(x)\,\mu_V(dx),
\end{split}
\end{equation*}
where in the inequality above we have used the fact that $\widetilde{\mathscr{E}}_{\alpha,V}(f,f) \leqslant \mathscr{E}_{\alpha,V}(f,f)$ for any $f\in C_b^\infty(\R^d).$

\medskip

Applying the local super Poincar\'{e} inequality (\ref{pr2.1.1}) into
the inequality above, we can obtain that for any $r\ge r_0$ and $f\in C_b^\infty(\R^d)$,
\begin{equation*}
\begin{split}
 \int f^2(x)\,\mu_V(dx)\leqslant& \Big(\frac{c_3}{\Phi(r)}+\frac{s}{2} \Big)\mathscr{E}_{\alpha,V}(f,f)\\
&+c_4\big(1+s^{-{d}/{\alpha}}\big)\frac{r^{d+{d^2}/{\alpha}}K(r)^{1+{d}/{\alpha}}}{k(r)^{2+{d}/{\alpha}}}
\Big(\int |f(x)|\,\mu_V(dx)\Big)^2,
\end{split}
\end{equation*} where we have used the fact that $\sup_{r\ge
r_0}\Phi(r)^{-1}<\infty$, thanks to (\ref{e12aa}).

If (\ref{e12aa}) holds, then $\lim_{r \rightarrow
\infty}\Phi(r)=\infty$.  By taking $r=\Phi^{-1}({2c_3}/{s})$ in the
estimate above, the required inequality (\ref{e13}) follows.
\end{proof}

To close this part, we present the following weighted Poincar\'{e}
inequality for $\mathscr{E}_{\alpha,V}$. The proof is similar to
that of \cite[Theorem 3.6]{CW}, and it is based on the local
Poincar\'{e} inequality \eqref{pr2.2.1} and Lemma \ref{le3.3}. We
omit the details here.

\begin{proposition}\label{pro-3.4}Under \eqref{th3.1.1}, there exists a constant $C_1>0$ such that
$$\int f^2(x)\Big(e^{V(x)}\inf_{|x|-1 \le |z|\le |x|-1/2}e^{-V(z)}\Big)\,\mu_V(dx)\le C_1\mathscr{E}_{\alpha,V}(f,f)$$ holds for all $f\in C_b^\infty(\R^d)$ with $ \mu_V(f)=0.$\end{proposition}
\subsection{Concentration of Measure about Functional Inequalities for $\mathscr{E}_{\alpha,V}$}\label{section3.2}
Recall that $V$ is a locally bounded measurable function on $\R^d$ such
that $e^{-V}\in L^1(dx)$, and $\mu_V(dx)=C_Ve^{-V(x)}\,dx$ is a
probability measure.

\begin{proposition}\label{prop1} $(1)$ Suppose that there exists a constant $C_1>0$ such that the Poincar\'{e} inequality holds
\begin{equation*}\mu_V(f^2)\leqslant C_1
\mathscr{E}_{\alpha,V}(f,f),\quad f \in C_b^{\infty}(\R^d), \mu_V(f)=0.
\end{equation*}
Then there exists a constant $\lambda_0>0$ such that
$$
\int e^{\lambda_0|x|}\,\mu_V(dx)<\infty.
$$

$(2)$ Assume that the following super Poincar\'{e} inequality holds
\begin{equation*}\mu_V(f^2)\le s \mathscr{E}_{\alpha,V}(f,f)+\beta(s) \mu_V(|f|)^2,\quad s>0,\,\,f \in C_b^{\infty}(\R^d),\end{equation*} where $\beta:(0,\infty)\to (0,\infty)$ is a decreasing function. Then, for any $\lambda>0$,
$$\int e^{\lambda|x|}\,\mu_V(dx)<\infty.$$

Furthermore, for each $r>0$, define
\begin{equation*}\label{prop1.4}
F(r):=\int_1^{\infty}e^{r\lambda}h(\lambda)\,d\lambda,
\end{equation*} where for every $\lambda>1$,
\begin{equation*}
h(\lambda):=\exp\bigg\{-(1+c_0)\lambda-
\lambda \int_1^{\lambda}\frac{1}{s^2}\log\bigg[2\beta\Big(
\frac{1}{ c_1s^2e^{2s}}\Big)\bigg]ds\bigg\},
\end{equation*} and
\begin{equation*}
c_0:=\log\bigg(\int e^{|x|}\,\mu_V(dx)\bigg),
\end{equation*}
$$c_1:=\int_{\{|z|\leqslant 1\}}\frac{dz}{|z|^{d+\alpha-2}}.$$
Then
$$
\int F(|x|)\mu_V(dx)<\infty.
$$
\end{proposition}

\begin{proof} (1) For any $n\ge1$, define $g_n(x):=e^{\lambda (|x|\wedge n)}$, where
$\lambda>0$ is a constant to be determined later. Clearly, $g_n$ is
a Lipschitz continuous bounded function. By the approximation
procedure in the proof of Theorem \ref{th3.1}, we can apply the
function $g_n$ into the Poincar\'{e} inequality. Thus,
\begin{equation*}
\aligned \int g_n^2(x)\,\mu_V(dx)&\leqslant
\frac{C_1}{2}\iint_{\{|x-y|\le 1\}}
\frac{(g_n(x)-g_n(y))^2}{|x-y|^{d+\alpha}}\,dy
\,\mu_V(dx)\\
&\qquad\qquad \qquad\qquad+\Big(\int
g_n(x)\,\mu_V(dx)\Big)^2.\endaligned
\end{equation*}

By the mean value theorem and the fact that for any $x$,
$y\in\R^d$, $n\ge1$, $$\big||x|\wedge n -|y|\wedge n\big|\leqslant
|x-y|,$$ we know that for any $x\in\R^d$,
\begin{equation*}
\begin{split}
\int_{\{|x-y|\le 1\}} \frac{(g_n(x)-g_n(y))^2}{|x-y|^{d+\alpha}}\,dy
=& \int_{\{|x-y|\leqslant 1\}}
\frac{(e^{\lambda (|x|\wedge n)}-e^{\lambda (|y|\wedge n)})^2}{|x-y|^{d+\alpha}}\,dy\\
 \leqslant&  \lambda^2 e^{2\big(\lambda (|x|\wedge
n)+\lambda \big)}
\int_{\{|x-y|\leqslant 1\}}\frac{|x-y|^2}{|x-y|^{d+\alpha}}\,dy\\
\leqslant & c_1\lambda^2 e^{2\big(\lambda (|x|\wedge n)+\lambda \big)}\\
 =&c_1 \lambda^2 e^{2\lambda }\, e^{2\lambda (|x|\wedge n)},
\end{split}
\end{equation*}
where $$c_1:=\int_{\{|z|\leqslant
1\}}\frac{dz}{|z|^{d+\alpha-2}}=\frac{d\pi^{d/2}}{(2-\alpha)\Gamma(d/2+1)}
.$$ Therefore,
\begin{equation}\label{e5a}
\begin{split}
\iint_{\{|x-y|\le 1\}}
\frac{(g_n(x)-g_n(y))^2}{|x-y|^{d+\alpha}}\,dy\,\mu_V(dx) \leqslant&
c_1 \lambda^2 e^{2\lambda }\,\int e^{2\lambda (|x|\wedge
n)}\,\mu_V(dx).
\end{split}
\end{equation}

For any $n\ge 1$ and $\lambda>0$, set $$l_n(\lambda):=\int
g_n^2(x)\,\mu_V(dx)=\int e^{2\lambda (|x|\wedge n)}\,\mu_V(dx).$$
Then, combining all the estimates above,
for each $\lambda>0$,
\begin{equation*}
\begin{split}
& l_n(\lambda)\leqslant \frac{C_1}{2}c_1\lambda^2 e^{2\lambda
}l_n(\lambda)+l_n^2({\lambda}/{2}).
\end{split}
\end{equation*}
Furthermore, using the Cauchy-Schwarz inequality, for any $R>0$, we
have
\begin{equation*}\label{proprop1.3.2}
\begin{split}
l_n^2({\lambda}/{2})&\leqslant \Big(e^{\lambda R }+
\int_{\{|x|>R\}}e^{\lambda (|x|\wedge n)}\,\mu_V(dx)\Big)^2\leqslant
2e^{2\lambda R}+2p(R)\,l_n(\lambda),
\end{split}
\end{equation*}
where $p(R):=\mu_V(|x|>R)$. Therefore, for each $R>0$ and
$\lambda>0$,
\begin{equation*}\label{proprop1.3.3}
l_n(\lambda)\leqslant \Big(\frac{C_1}{2}c_1\lambda^2 e^{2\lambda
}+2p(R)\Big)l_n(\lambda)+2e^{2\lambda R},
\end{equation*}

Now, we fix $R_0>0$ large enough such that $p(R_0)<{1}/{8}$, and
then take $\lambda_0>0$ small enough such that $C_1c_1\lambda_0^2
e^{2\lambda_0}<{1}/{2}$. Then, we arrive at
\begin{equation*}
l_n(\lambda_0)\leqslant 4e^{2\lambda_0 R_0}.
\end{equation*}
Letting $n \rightarrow \infty$, we obtain the first desired assertion.

(2) We still use the same test function $g_n$ as that in part (1). By applying this test
function $g_n$ into the super Poincar\'{e} inequality and by using
(\ref{e5a}), we have
\begin{equation}\label{e6a}
\aligned \int g_n^2(x)\,\mu_V(dx)&\leqslant
\frac{c_1}{2}\lambda^2 e^{2\lambda}s\int g_n^2(x)\,\mu_V(dx)\\
&\qquad\qquad \qquad\qquad+\beta(s)\Big(\int
g_n(x)\,\mu_V(dx)\Big)^2,\quad s>0.\endaligned
\end{equation}
Following the argument in the proof of part (1), we can get that for any
$\lambda$, $s$ and $R>0$,
$$l_n(\lambda)\le \frac{c_1}{2}s\lambda^2 e^{2\lambda}l_n(\lambda)+\beta(s)\Big[2e^{2\lambda R}+2p(R)l_n(\lambda)\Big],$$
where $l_n(\lambda)$ and $p(R)$  are the same functions defined in the
proof of part (1).

Now, for any fixed $\lambda>0$, choose $s_0>0$ small enough such
that $c_1s_0\lambda^2 e^{2\lambda}<1/2$, and then take $R_0$ large
enough such that $\beta(s_0)p(R_0)<1/8$, we get
$$l_n(\lambda)\le 8\beta(s_0)e^{2\lambda R_0}.$$ Letting $n \rightarrow \infty$, we show
$\int e^{\lambda|x|}\,\mu_V(dx)<\infty$ for any $\lambda>0$.

In the remainder of this part, we will follow the method adopted in the proof of  \cite[Theorem 3.3.20]
{WBook}, see also \cite[Theorem 6.1]{W1}.
For every $\lambda>0$, set $l(\lambda):=\mu_V(e^{\lambda |x|})$. For any
$\varepsilon>0$, it holds that
\begin{equation*}
\begin{split}
l'(\lambda)&=\mu_V(|x|e^{\lambda|x|})\\
&=\mu_V\Big[\Big( \frac{1}{\lambda}(\lambda |x|+\log \varepsilon)-\frac{\log \varepsilon}{\lambda}\Big)e^{\lambda|x|}\Big]\\
&=\mu_V\Big( \frac{1}{\lambda}(\lambda |x|+\log \varepsilon)e^{\lambda|x|}\Big)-
\frac{\log \varepsilon}{\lambda} \mu_V(e^{\lambda|x|})\\
&\le \varepsilon \mu_V(e^{2\lambda |x|})-
\frac{\text{log}(\varepsilon \lambda e)}{\lambda}\mu_V(e^{\lambda |x|})\\
&= \varepsilon l(2\lambda)-
\frac{\text{log}(\varepsilon \lambda e)}{\lambda}l(\lambda),
\end{split}
\end{equation*}
where in the inequality above we have applied the Young inequality
$$st \le s\text{log}s -s+e^t,\quad s \in \R_+,\,\,  t \in \R$$ with
$s=\frac{1}{\lambda}$ and $t=\lambda |x|+\log\varepsilon$.

On the other hand, according to (\ref{e6a}) and letting $n
\rightarrow \infty$,
\begin{equation*}
l(2\lambda) \le \frac{c_1}{2}\lambda^2e^{2\lambda}s
l(2\lambda)+\beta(s)l(\lambda)^2,\quad s>0.
\end{equation*}
 Taking $s=\big( c_1 \lambda^2 e^{2\lambda}\big)^{-1}$, we obtain that
$$l(2\lambda)\le 2\beta\Big(\frac{1}{c_1 \lambda^2 e^{2\lambda}}\Big)
l(\lambda)^2.$$ Combining all the estimates above, $$l'(\lambda)\le
2\varepsilon \beta\Big(\frac{1}{ c_1 \lambda^2 e^{2\lambda}}\Big)
l(\lambda)^2- \frac{\text{log}(\varepsilon \lambda
e)}{\lambda}l(\lambda).$$

Choosing $\varepsilon= \Big(2\lambda
l(\lambda)\beta\big(\frac{1}{c_1  \lambda^2 e^{2\lambda}
}\big)\Big)^{-1}$, we derive
\begin{equation*}
l'(\lambda)\le \frac{l(\lambda)}{\lambda}\bigg[\log
l(\lambda)+\log\Big(2\beta \Big(\frac{1}{c_1 \lambda^2
e^{2\lambda}}\Big)\Big)\bigg],
\end{equation*}
hence
\begin{equation*}
\frac{d}{d\lambda}\Big( \frac{\log l(\lambda)}{\lambda}\Big)\le
\frac{1}{\lambda^2}\log\Big(2\beta \Big(\frac{1}{c_1 \lambda^2
e^{2\lambda}}\Big)\Big),
\end{equation*}
which implies that for any $\lambda\ge1,$
$$ l(\lambda)\le  \exp\bigg(\lambda \log l(1)+ \lambda \int_1^\lambda\frac{1}{s^2}\log\Big(2\beta
\Big(\frac{1}{c_1 s^2 e^{2s}}\Big)\,ds\bigg).$$ Then, by
the Fubini theorem, we have
$$
\int F(|x|)\,\mu_V(dx)=\int_1^{+\infty}\int e^{\lambda
|x|}\,\mu_V(dx)h(\lambda)\,d\lambda\le \int_1^\infty e^{-\lambda}\,d\lambda<\infty.
$$ This finishes the proof.
\end{proof}

Now, we turn to the proof of Example \ref{Ex1}.

\begin{proof}[Proof of Example \ref{Ex1}]
(1) Let $\mu_{V_\lambda}(dx)=C_{\lambda}e^{-\lambda |x|}\,dx=:C_{\lambda}e^{-V_{\lambda}(x)}\,dx$ with $\lambda>0$, we have
\begin{equation*}
\frac{\inf_{|x|-1 \le |z|\leqslant
|x|-{1}/{2}}e^{-V_{\lambda}(z)}} {\sup_{|x| \le |z|\le |x|+1}e^{-V_{\lambda}(z)}}\ge
e^{\lambda/2}, \quad x\ge1.
\end{equation*} Then, for $\lambda_0$ defined in Example \ref{Ex1} (1), if $\lambda>\lambda_0$,
\begin{equation*}
\liminf_{|x| \rightarrow \infty}\frac{\inf_{|x|-1 \le |z|\leqslant
|x|-{1}/{2}}e^{-V_{\lambda}(z)}} {\sup_{|x| \le |z|\le |x|+1}e^{-V_{\lambda}(z)}}>
\frac{1}{\alpha}2^{2d+1}(e+e^{{1}/{2}})(2^{\alpha}-1).
\end{equation*}
According to Theorem \ref{th3.1} (1), the Poincar\'{e} inequality (\ref{e12}) holds for
$\mu_{V_\lambda}(dx)$ with $\lambda>\lambda_0$.

(2) If the
super Poincar\'{e} inequality (\ref{e13}) holds for $\mu_{V_\delta}(dx)=C_{\delta}e^{- (1+|x|^{\delta})}\,dx=:C_{\delta}e^{-V_{\delta}(x)}\,dx$, then, by Proposition \ref{prop1} (2), $\int e^{\lambda |x|}\mu_{V_\delta}(dx)<\infty$
for any $\lambda>0$, which implies that the
super Poincar\'{e} inequality (\ref{e13}) holds only if $\delta>1$.

On the other hand, for every
$\delta>1$ and for $|x|$ large enough,
\begin{equation*}
e^{V_{\delta}(x)}\inf_{|z|\leqslant
|x|-{1}/{2}}e^{-V_{\delta}(z)}\ge C_1e^{C_2|x|^{\delta-1}},
\end{equation*}
where $C_1$ and $C_2$ are two positive constants independent of $x$. Hence, for $r$ large enough,
$\Phi(r)\ge C_1e^{C_2 r^{\delta-1}}$.  Therefore, according to Theorem \ref{th3.1} (2), we know that the super
Poincar\'{e} inequality (\ref{e13}) holds with the rate function $\beta$ given by
(\ref{ex1.2.1a}).

According to \cite[Theorem 3.3.14]{WBook} (also see \cite[Theorem
5.1]{W1}), if the rate function $\beta(s)$ satisfies that
\begin{equation}\label{e11aa}
\Psi(t):=\int_t^{\infty}\frac{\beta^{-1}(r)}{r}\,dr<\infty \quad\textrm{ for any }t > \inf_{s>0}\beta(s),
\end{equation}
then
\begin{equation}\label{e11a}
\|P_t^{\alpha,V_\delta}\|_{L^1(\mu_{V_\delta})\to L^\infty(\mu_{V_\delta})} \le 2\Psi^{-1}(t),\quad t>0,
\end{equation}
where $$\Psi^{-1}(t):=\inf \Big \{r \ge \inf_{s>0}\beta(s):\ \Psi(r) \ge t \Big\}.$$
It follows from (\ref{ex1.2.1a}) that $$\beta^{-1}(r)\le \exp\{-C_3(\log r +C_4)^{\frac{\delta-1}{\delta}}\}$$ holds for $r$ large enough and some positive constants $C_3$ and $C_4$. Hence, for $t$ large enough,
\begin{equation*}
\begin{split}
\Psi(t)&\le \int_t^{\infty}\frac{1}{r\exp\{C_3(\log r +C_4)^{\frac{\delta-1}{\delta}}\}}\,dr\\
&\le \int_t^{\infty}\frac{1}{r(\log r+C_4)^{\frac{1}{\delta}}\exp\{C_5(\log r +C_4)^{\frac{\delta-1}{\delta}}\}}\,dr\\
&= \frac{C_6}{\exp\{C_5(\log t +C_4)^{\frac{\delta-1}{\delta}}\}}.
\end{split}
\end{equation*}
This along with (\ref{e11a}) gives us the desired estimate for the associated semigroup $P_t^{\alpha,V_\delta}$.

Furthermore, assume that the super
Poincar\'{e} inequality (\ref{e13}) holds for $\mu_{V_\delta}$ with the rate function
$\beta(s)$ satisfying (\ref{ex1.2.1}). Then for any $\varepsilon>0$ small enough, there is
a $s_0:=s_0(\varepsilon)>0$ such that
for any $s\le s_0$,
\begin{equation*}
\log \beta(s)\le \varepsilon \log^{\frac{\delta}{\delta-1}}(1+s^{-1}).
\end{equation*}
Hence, there is a constant $C_7>0$ (independent of $\varepsilon$) such that
for every $\varepsilon>0$ and $s\ge1$,
\begin{equation*}
\log\Big(2\beta
\Big(\frac{1}{c_1 s^2 e^{2s}}\Big)\Big)\le C_7\varepsilon s^{\frac{\delta}{\delta-1}}+C_8(\varepsilon),
\end{equation*} where $C_8(\varepsilon)>0$ may depend on $\varepsilon$.
Let $F(r)$ be the function defined in Proposition \ref{prop1} (2). Therefore,
for every $r>0$ large enough and $\varepsilon>0$ small enough,
\begin{equation*}
\begin{split}
 F(r)&\ge \int_1^{\infty} \exp\Big\{r\lambda-(c_0+1)\lambda-\lambda\int_1^{\lambda}\frac{1}{s^2}\big(C_7\varepsilon
s^{\frac{\delta}{\delta-1}}+C_8(\varepsilon)\big)\,ds\Big\}\,d\lambda\\
&\ge \int_1^{\infty} \exp\Big\{-C_9\varepsilon \lambda^{\frac{\delta}{\delta-1}}+(r-C_{10}(\varepsilon))\lambda\Big\}\,d\lambda\\
&\ge \int_1^{\big(\frac{r}{2C_9\varepsilon}\big)^{\delta-1}}e^{(\frac{r}{2}-C_{11}(\varepsilon))\lambda}\,d\lambda,
\end{split}
\end{equation*}
where in the last inequality we have used the fact that if $\lambda \le (\frac{r}{2C_9\varepsilon})^{\delta-1}$, then $C_9\varepsilon\lambda^{\frac{1}{\delta-1}}\le {r}/{2}.$ The inequality above shows that, for any $\varepsilon>0$ small enough there are two constants $C_{12}>0$ (independent of $\varepsilon$ and $r$) and $C_{13}(\varepsilon)>0$ (independent of $r$) such that for $r>0$ large enough,
\begin{equation*}
F(r)\ge \frac{C_{13}(\varepsilon)}{r}\exp\Big[\Big(\frac{C_{12}}{\varepsilon}\Big)^{\delta-1}r^{\delta}\Big].
\end{equation*}
This, along with Proposition \ref{prop1} (2), yields that for any $\kappa>0$,
$$
\int e^{\kappa|x|^{\delta}}\,\mu_{V_\delta}(dx)<\infty,
$$
However, the statement above can not be true since $\mu_{V_\delta}(dx)=C_{\delta}e^{- (1+|x|^{\delta})}\,dx$. That is, there is a contradiction, so the super Poincar\'{e} inequality (\ref{e13}) does not hold
for $\mu_{V_\delta}$ with the rate function $\beta(s)$ satisfying (\ref{ex1.2.1}).

(3)  Let $\mu_{V_\theta}(dx)=C_{\theta}e^{- |x|\log^{\theta}(1+|x|)}\,dx=:C_{\theta}e^{-V_{\theta}(x)}\,dx$. Suppose that in this case the
super Poincar\'{e} inequality (\ref{e13}) holds. Then, according to Proposition \ref{prop1}, $\int e^{\lambda |x|}\,\mu_{V_\theta}(dx)<\infty$
for any $\lambda>0$, which implies the
super Poincar\'{e} inequality (\ref{e13}) holds for $\mu_{V_\theta}$ only with $\theta>0$.

On the other hand, for every
$\theta>0$, there exist two positive constants $C_1$ and $C_2$ such that for $|x|$ large enough,
\begin{equation*}
e^{V_{\theta}(x)}\inf_{|z|\leqslant
|x|-{1}/{2}}e^{-V_{\theta}(z)}\ge C_1e^{C_2\log^{\theta}(1+|x|)}.
\end{equation*}
Then, for $r$ large enough, we have
 $\Phi(r)\ge C_1 e^{C_2\log^{\theta}(1+r)}$. Therefore, by Theorem \ref{th3.1} (2), we can get that the super Poincar\'{e} (\ref{e13}) holds for $\mu_{V_\theta}$ with the rate function
$\beta(s)$ given by (\ref{ex1.2.2a}).

On the other hand, according to (\ref{ex1.2.2a}), we have $$\beta^{-1}(r)\le \exp\big\{-C_3\log^{\theta}\big(C_4(\log r+C_5)\big)\big\}$$
for $r$ large enough and some positive constants $C_i$ $(i=3,4,5)$. Let $\Psi(t)$ be the function defined by (\ref{e11aa}). Then, for $t$ large
enough, we have
\begin{equation*}
\begin{split}
\Psi(t) &\le \int_t^{\infty} \frac{1}{r\exp\big\{C_3\log^{\theta}\big(C_4(\log r+C_5)\big)\big\}}\,dr\\
& \le  \int_t^{\infty} \frac{\log^{\theta-1}\big(C_4(\log r +C_5)\big)}
{r(\log r +C_5)\exp\big\{C_6\log^{\theta}\big(C_4(\log r+C_5)\big)\big\}}\,dr\\
&=\frac{C_7}{\exp\big\{C_6\log^{\theta}\big(C_4(\log t+C_5)\big)\big\}},
\end{split}
\end{equation*}
where in the second inequality we have used the fact that if $\theta>1$, then
\begin{equation*}
\exp\big\{C_3\log^{\theta}\big(C_4(\log r+C_5)\big)\big\} \ge (\log r +C_5)\exp\big\{C_6\log^{\theta}\big(C_4(\log r+C_5)\big)\big\}
\end{equation*}
holds for $r$ large enough and some positive constants $C_3>C_6$. Combining the estimate above with (\ref{e11a}), we get the desired estimate for the associated semigroup $P_t^{\alpha,V_\theta}.$

Next, we assume that (\ref{e13}) holds for $\mu_{V_\theta}$ with the rate function
$\beta(s)$ satisfying (\ref{ex1.2.2}). Then for any $\varepsilon>0$ small enough, there is a constant $s_0:=s_0(\varepsilon)>0$ such that
for any $s\le s_0$,
\begin{equation*}
\log \beta(s)\le \exp\Big\{\varepsilon \log^{\frac{1}{\theta}}(1+s^{-1})\Big\}.
\end{equation*}
Hence for every $s\ge1$ and $\varepsilon>0$ small enough,
\begin{equation*}
\text{log}\Big(2\beta
\big(\frac{1}{c_1  s^2 e^{2s}}\big)\Big)\le
\exp\Big\{C_8\varepsilon s^{\frac{1}{\theta}}+C_9(\varepsilon)\Big\},
\end{equation*}
where $C_8>0$ is independent of $\varepsilon$, and $C_9(\varepsilon)>0$ may depend on $\varepsilon$.
Therefore, by the similar argument in the proof of part (2), for $r>0$ large enough and $\varepsilon>0$ small enough,
\begin{equation*}
\begin{split}
 F(r)&\ge \int_1^{\infty} \exp\Big\{r\lambda-(c_0+1)\lambda-\lambda\int_1^{\lambda}\frac{1}{s^2}
 \exp\Big\{C_8\varepsilon s^{\frac{1}{\theta}}+C_9(\varepsilon)\Big\}\,ds\Big\}\,d\lambda\\
&\ge \int_1^{\infty} \exp\Big\{-C_{10}(\varepsilon)\lambda e^{\varepsilon C_8\lambda^{\frac{1}{\theta}}}
+(r-C_{11}(\varepsilon))\lambda\Big\}\,d\lambda\\
&\ge \int_1^{\big(\frac{\log r -\log (2C_{10}(\varepsilon))}{C_8 \varepsilon}\big)^{\theta}}
e^{(\frac{r}{2}-C_{11}(\varepsilon))\lambda}\,d\lambda\\
&\ge \frac{C_{13}(\varepsilon)}{r}\exp\Big\{\frac{C_{12}}{\varepsilon^{\theta}}r\log^{\theta} r\Big\},
\end{split}
\end{equation*}
where $C_{12}>0$ is independent of $\varepsilon,r$, and $C_{13}(\varepsilon)>0$ is independent of $r$.
Thus, according to Proposition \ref{prop1} (2), for
any $\kappa>0$,
\begin{equation*}
\int e^{\kappa|x|\log^{\theta}(1+|x|)}\,\mu_{V_\theta}(dx)<\infty,
\end{equation*}
which however can not be true, since
$\mu_{V_\theta}(dx)=C_{\theta}e^{- |x|\log^{\theta}(1+|x|)}\,dx$.
This contradiction shows that the super Poincar\'{e} inequality
(\ref{e13}) does not hold for $\mu_{V_\theta}$ with the rate
function $\beta(s)$ satisfying (\ref{ex1.2.2}).\end{proof}

\subsection{Comparison of the Functional Inequalities for $\mathscr{E}_{\alpha,V}$ and
$D_{\rho,V}$} In this subsection, we aim to compare the criteria for
the Poincar\'{e} inequality and the super Poincar\'{e} inequality
between $\mathscr{E}_{\alpha,V}$ and $D_{\rho,V}$. First, we take
$\rho(r)=r^{-d-\alpha}e^{-\delta r}$ with $\alpha\in(0,2)$ and
$\delta\ge0$ in \eqref{dir1}, and set
$$D_{\alpha, \delta, V}(f,f):=\frac{1}{2}\iint\frac{(f(x)-f(y))^2}{|x-y|^{d+\alpha}}e^{-\delta|x-y|}\,dy\,\mu_V(dx).$$
We denote $D_{\alpha,0,V}$ by $D_{\alpha,V}$ for simplicity. Theorem
\ref{th3.1} yields the following
\begin{corollary} \label{corollary1}
Let $\alpha\in(0,2)$ and $\delta\in [0,\infty)$. For  any $a>0$, set
$\tilde V_a(x):=V(ax)$.

$(1)$ If there is a constant $a>0$ such that
\begin{equation}\label{e9a}
\liminf_{|x|\rightarrow \infty}\frac{\inf_{|x|-1 \le |z|\leqslant
|x|-{1}/{2}}e^{-\tilde V_a(z)}} {\sup_{|x| \le |z|\le
|x|+1}e^{-\tilde
V_a(z)}}>\frac{1}{\alpha}2^{2d+1}(e+e^{{1}/{2}})(2^{\alpha}-1),
\end{equation}
then there is a constant $c_1>0$ such that for any
$f\in C_b^\infty(\R^d)$,
\begin{equation*}
\begin{split}
\mu_V\big((f-\mu_V(f))^2\big) \leqslant c_1 D_{\alpha,\delta,V}(f,f).
\end{split}
\end{equation*}

$(2)$ Suppose there is a constant $a>0$ such that
\begin{equation}\label{e9b}\liminf_{|x|\rightarrow \infty}\frac{\inf_{|x|-1 \le |z|\leqslant
|x|-{1}/{2}}e^{-\tilde V_a(z)}} {\sup_{|x| \le |z|\le
|x|+1}e^{-\tilde V_a(z)}}=\infty.\end{equation} Let $\tilde
\beta_a(s)$ be the rate function defined by \eqref{e13a} with
$\tilde V_a(x):=V(ax)$ in place of $V(x)$. If moreover there is a
constant $c_2>0$ such that
\begin{equation}\label{e9c}
\tilde \beta_a(s)\le \exp\Big(c_2\big(1+{s}^{-1}\big)\Big), \quad
s>0,
\end{equation}
then the following log-Sobolev inequality holds
\begin{equation*}
\mu_V(f^2\log f^2)-\mu_V(f^2)\log\mu_V(f^2) \le
c_3D_{\alpha,\delta,V}(f,f),\quad f \in C_b^\infty(\R^d).
\end{equation*}
\end{corollary}
\begin{proof}

(a) For any function $f\in C_b^{\infty}(\R^d)$ with $\int f
\,d\mu_V=0$, define $\tilde f(x):=f(ax)$ for all $x \in \R^d$. By
changing the variable, it is easy to check that $\int \tilde f
\,d\mu_{\tilde V_a}=0$. According to (\ref{e9a}) and Theorem
\ref{th3.1} (1), we know that
 \begin{equation*}
\int \tilde f^2(x)\,\mu_{\tilde V_a}(dx)\le \frac{C_0}{2}
\iint_{\{|x-y|\le 1\}}\frac{(\tilde f(x)-\tilde f(y))^2}
{|x-y|^{d+\alpha}}\,dy\,\mu_{\tilde V_a}(dx)
\end{equation*}
holds for some constant $C_0>0$ independent of $f$. Then, by changing the variable again, we arrive at
\begin{equation*}
\int  f^2(x)\,\mu_{V}(dx)\le \frac{a^{\alpha}C_0}{2} \iint_{\{|x-y|\le a\}}\frac{( f(x)-f(y))^2}
{|x-y|^{d+\alpha}}\,dy\,\mu_{ V}(dx).
\end{equation*}
Combining this inequality with the fact that
\begin{equation}\label{e15a}
\frac{1}{2}\iint_{\{|x-y|\le a\}}\frac{( f(x)- f(y))^2}
{|x-y|^{d+\alpha}}\,dy\,\mu_{ V}(dx) \le e^{a\delta} D_{\alpha, \delta, V}(f,f),
\end{equation}
we can get the first required conclusion.

(b) Suppose that \eqref{e9b} holds and the rate function
$\tilde\beta_a(s)$ defined by (\ref{e13a}) with respect to $\tilde
V_a(x)$ satisfies (\ref{e9c}). By Theorem \ref{th3.1} (2) and
\cite[Corollary 3.3.4]{WBook} (see also \cite[Corollary 3.3]{W1}),
the following defective log-Sobolev inequality holds for any $g\in
C_b^{\infty}(\R^d)$,
\begin{equation}\label{e15}
\begin{split}
\mu_{\tilde V_a}(g^2\log g^2)-&\mu_{\tilde V_a}(g^2)\log\mu_{\tilde V_a}(g^2)\\
& \le C_1 \iint_{\{|x-y|\le 1\}}\frac{(g(x)-g(y))^2}
{|x-y|^{d+\alpha}}\,dy\,\mu_{\tilde V_a}(dx)+C_2\mu_{\tilde
V_a}(g^2),
\end{split}
\end{equation}
where $C_1$ and $C_2$ are two positive constants. Hence, for any
$f\in C_b^{\infty}(\R^d)$, by applying $\tilde f(x):=f(ax)$ into
(\ref{e15}) and by the change of variable and (\ref{e15a}), we get
that
\begin{equation}\label{e16}
\begin{split}
\mu_{V}(f^2\log f^2)-&\mu_{V}(f^2)\log\mu_{V}(f^2)\\
& \le 2a^{\alpha}e^{a\delta}C_1 D_{\alpha,\delta,V}(f,f)+(C_2-d\log
a)\mu_{V}(f^2).
\end{split}
\end{equation}

  If $C_2-d\log a\le 0$, then, by (\ref{e16}), we get the second required conclusion. If
$C_2-d\log a>0 $, then (\ref{e16}) indeed is a defective log-Sobolev
inequality. On the other hand, according to \eqref{e9b} and (1), we
know that the Poincar\'{e} inequality holds for
$D_{\alpha,\delta,V}(f,f)$, which along with (\ref{e16}) yields the
real log-Sobolev inequality, e.g.\ see \cite[Theorem 5.1.8]{WBook}.

\end{proof}
Corollary \ref{corollary1} improves \cite[Theorem 1.1]{CW} for ${D}_{\alpha,\delta,V}$ when
$\delta>0$ large enough. The detail also can be seen from the
following example.

\begin{example}\label{exm0} (1)
Let $\mu_V(dx):=\mu_\lambda(dx)=C_{\lambda}e^{-\lambda|x|}\,dx$ with
$\lambda>0$. Then, (\ref{e9a}) is satisfied for such $\mu_V$, and
hence the Poincar\'{e} inequality holds for $D_{\alpha,\delta,V}$
with any $\delta\ge0$; while \cite[Theorem 1.1]{CW} only yields that
the Poincar\'{e} inequality holds for $D_{\alpha,\delta,V}$ with
$\delta\in[0,\lambda]$.

(2) Let $\mu_V(dx):=C_{\lambda}e^{-\lambda|x|\log(1+|x|)}\,dx$ with
$\lambda>0$. Then,  (\ref{e9b}) and \eqref{e9c} hold for such
$\mu_V$, and hence the log-Sobolev inequality holds for
$D_{\alpha,\delta,V}$ with any $\delta\ge0$.
\end{example}

\begin{remark}\label{exm0.2} Indeed, according to the arguments of Example \ref{Ex1} and Corollary
\ref{corollary1}, we can find the following two statements: (i) Let
$\mu_{V_{\lambda}}(dx):=C_{\lambda}e^{-\lambda|x|}\,dx$ with
$\lambda>0$. Then, there are two positive constants $a_1$ and $C_1$
(may depend on $\lambda$) such that
\begin{equation*}
\mu_{V_{\lambda}}\big(f^2)-\mu_{V_{\lambda}}(f)^2\le
C_1\iint_{\{|x-y|\le a_1\}}\frac{( f(x)- f(y))^2}
{|x-y|^{d+\alpha}}\,dy\,\mu_{ V_{\lambda}}(dx),\quad\ f\in
C_b^{\infty}(\R^d).
\end{equation*}
(ii) Let $\mu_{ \hat
V_{\lambda}}(dx):=C_{\lambda}e^{-\lambda|x|\log(1+|x|)}\,dx$ with
$\lambda>0$. Then, there are two positive constants $a_2$ and $C_2$
(may depend on $\lambda$) such that
\begin{equation*}
\begin{split}
\mu_{\hat V_{\lambda}}\big(f^2\log f^2\big)-&
\mu_{\hat V_{\lambda}}(f^2)\log \mu_{\hat V_{\lambda}}(f^2)\\
&\le C_2\iint_{\{|x-y|\le a_2\}}\frac{( f(x)- f(y))^2}
{|x-y|^{d+\alpha}}\,dy\,\mu_{ \hat V_{\lambda}}(dx),\quad \ f\in
C_b^{\infty}(\R^d).
\end{split}
\end{equation*}
In particular, a close inspection of the computation in Example
\ref{Ex1} shows that, if $\lambda$ is large enough then one can take
both the jump sizes $a_1$ and $a_2$ in two inequalities above to be
less than 1; however, for small $\lambda$ we can not expect the jump
sizes $a_1$ and $a_2$ to be less than $1$.
\end{remark}

\medskip

 To compare the different properties of the functional inequalities for
 $D_{\alpha,\delta,V}$ and $\mathscr{E}_{\alpha,V}$, we will take the following three examples.

\begin{example}\label{exm1} [{\bf Poincar\'{e} inequalities and super Poincar\'{e} inequalities hold for ${D}_{\alpha,V}$ but not for ${D}_{\alpha,\delta,V}$ with $\delta>0$ and $\mathscr{E}_{\alpha,V}$.}]
Let
$\mu_{V_{\varepsilon}}(dx):=C_\varepsilon(1+|x|)^{d+\varepsilon}\,dx$
with $\varepsilon\ge\alpha$. Then, according to \cite[Corollary
1.2]{WW}, the Poincar\'{e} inequality holds for
$D_{\alpha,V_{\varepsilon}}$ with $\varepsilon\ge \alpha$; and the
super Poincar\'{e} inequality holds for $D_{\alpha,V_{\varepsilon}}$
with $\varepsilon> \alpha$. However, by \cite[Proposition 1.3]{CW}, the Poincar\'{e}
inequality and so the super Poincar\'{e} inequality do not hold for
${D}_{\alpha,\delta, V_\varepsilon}$ with any $\delta>0$. On the other hand, according to Proposition \ref{prop1}, Poincar\'{e} inequality and the super Poincar\'{e} inequality either do not hold for
$\mathscr{E}_{\alpha,V_{\varepsilon}}$.
\end{example}

\begin{example}\label{exm2} [{\bf Super Poincar\'{e} inequalities hold for
${D}_{\alpha,\delta,V}$ with $\delta>0$ but not
$\mathscr{E}_{\alpha,V}$.}]
Let $\mu_{V_{\lambda}}(dx):=C_{\lambda}e^{-\lambda|x|}\,dx$ with $\lambda>0$. For
every $0<\delta<\lambda$, according to \cite[Lemma 4.3]{CW} and the
argument of \cite[Theorem 1.1 (2)]{WW}, the super Poincar\'{e}
inequality holds for $D_{\alpha, \delta, V_{\lambda}}$ with the rate
function $\beta(s)=c_1(1+s^{-p_1})$
for some positive
constants $c_1$ and $p_1$. However, by Proposition \ref{prop1}, the
super Poincar\'{e} inequality does not hold for
$\mathscr{E}_{\alpha,V_{\lambda}}$.
\end{example}

\begin{example}\label{exm3} [{\bf Super Poincar\'{e} inequalities hold for both
${D}_{\alpha,\delta,V}$ and
$\mathscr{E}_{\alpha,V}$, but with different rate function.}]
Let $\mu_{V_{\kappa}}(dx):=C_{\kappa}e^{-(1+|x|^{\kappa})}\,dx$ with
$\kappa>1$. Also according to \cite[Lemma 4.3]{CW} and the argument
 of \cite[Theorem 1.1 (2)]{WW}, the super Poincar\'{e} inequality holds for
$D_{\alpha, \delta, V_{\kappa}}$ with the rate function
$\beta(s)=c_2(1+s^{-p_2})$
for some positive constants
$c_2$ and $p_2$. On the other hand, according to Example \ref{Ex1}
(2), the super Poincar\'{e} holds for
$\mathscr{E}_{\alpha,V_{\kappa}}$ with the rate function
\begin{equation*}
\beta(s)=c_3\exp\Big(c_4\big(1+\log^{{\kappa}/{(\kappa-1)}}(1+{s}^{-1})\big)\Big).
\end{equation*}
\end{example}

\section{Functional Inequalities for Non-local Dirichlet Forms with large Jumps}\label{section4}
\begin{proof}[Proof of Theorem \ref{th3.2}]
(1) The proof of \eqref{th3.2.2}  is almost the same as that of \cite[Theorem 3.6]{CW}. For reader's convenience,
here we write it in detail. Let $ L_{\mathscr{D}_{\alpha,V}} $ be the generator associated with $\mathscr{D}_{\alpha,V}$.
Then, according to \cite[Lemma 4.2]{CW}, we know that for any $f\in C_c^\infty(\R^d)$,
$$L_{\mathscr{D}_{\alpha,V}}f(x)=\frac{1}{2}\int_{\{|z|>1\}} \big (f(x+z)-f(x)\big)
\big(e^{V(x)-V(x+z)}+1\big)\,\frac{dz}{|z|^{d+\alpha}}.$$
For $\alpha_0\in (0,1)$, let $\phi\in C^{\infty}(\R^d)$ such that $\phi \geqslant 1$ and $\phi(x)=1+|x|^{\alpha_0}$
for $|x|>1$.
By (\ref{th3.2.1}) and \cite[Lemma 4.3]{CW}, $L_{\mathscr{D}_{\alpha,V}}\phi$ is well defined, and there exist $r_0$, $c_1$ and $c_2>0$ such
that
$$
L_{\mathscr{D}_{\alpha,V}}\phi(x)\leqslant -c_1\frac{e^{V(x)}}{(1+|x|)^{d+\alpha}}\phi(x)+c_2\I_{B(0,r_0)}(x).
$$
This, along with \cite[Propisition 3.2]{CW}, yields that there are $c_3, c_4>0$ such that for any $f\in
C_b^\infty(\R^d)$,
\begin{equation*}
\int f(x)^2\frac{e^{V(x)}}{(1+|x|)^{d+\alpha}}\,\mu_V(dx)\leqslant c_3 \mathscr{D}_{\alpha,V}(f,f)+c_4\int_{B(0,r_0)}f^2\phi^{-1}\,d\mu_V.
\end{equation*} In particular, for any $f\in C_b^\infty(\R^d)$ with $\mu_V(f)=0$,  \begin{equation*}
\int f(x)^2\frac{e^{V(x)}}{(1+|x|)^{d+\alpha}}\,\mu_V(dx)\leqslant c_3 \mathscr{D}_{\alpha,V}(f,f)+c_4\int_{B(0,r_0)}f^2\phi^{-1}\,d\mu_V.
\end{equation*}

On the other hand, since $\phi\geqslant 1$, by the local Poincar\'{e}
inequality \eqref{pr2.4.1}, there is a constant $c_5>0$ such that for any $r>r_0\vee 3$,
\begin{align*}\int_{B(0,r_0)}f^2\phi^{-1}\,d\mu_V&\leqslant \int_{B(0,r_0)}f^2\,d\mu_V\\
&\leqslant \int_{B(0,r)}f^2\,d\mu_V\\
&\le \frac{c_5K(r)r^{2d+\alpha}}{k(r)}\mathscr{D}_{\alpha,V}(f,f)+\frac{1}{\mu_V(B(0,r))}\Big(\int_{B(0,r)}f\,d\mu_V\Big)^2\\
&=\frac{c_5K(r)r^{2d+\alpha}}{k(r)}\mathscr{D}_{\alpha,V}(f,f)+\frac{1}{\mu_V(B(0,r))}\Big(\int_{B(0,r)^c}f\,d\mu_V\Big)^2,\end{align*}
where in the equality above we have used the fact that
$$\int_{B(0,r)}f\,d\mu_V=-\int_{B(0,r)^c}f\,d\mu_V.$$
Using the Cauchy-Schwarz inequality, we find
\begin{align*}\Big(\int_{B(0,r)^c}f\,d\mu_V\Big)^2\leqslant &\Big(\int_{B(0,r)^c} \frac{(1+|x|)^{d+\alpha}}{e^{V(x)}}\,\mu_V(dx)\Big)\int_{B(0,r)^c}f(x)^2\frac{e^{V(x)}}{(1+|x|)^{d+\alpha}}\,\mu_V(dx).\end{align*}

Therefore, for any $f\in C_b^\infty(\R^d)$ with $\int f\,d\mu_V=0$
and any $r\geqslant r_0\vee 3$,
\begin{align*}&\int f(x)^2\frac{e^{V(x)}}{(1+|x|)^{d+\alpha}}\,\mu_V(dx)\\
&\leqslant \left(c_3+\frac{c_6K(r)r^{2d+\alpha}}{k(r)}\right)\mathscr{D}_{\alpha,V}(f,f)\\
&\quad+\frac{c_6\int_{B(0,r)^c} \frac{(1+|x|)^{d+\alpha}}{e^{V(x)}}\,\mu_V(dx) }{\mu_V(B(0,r))}\int f(x)^2\frac{e^{V(x)}}{(1+|x|)^{d+\alpha}}\,\mu_V(dx).\end{align*}
Due to \eqref{th3.2.1}, $\int \frac{(1+|x|)^{d+\alpha}}{e^{V(x)}}\,\mu_V(dx)<\infty$, and so we can choose $r_1\ge r_0\vee3$ large
enough such that $$\frac{c_6\int_{B(0,r_1)^c} \frac{(1+|x|)^{d+\alpha}}{e^{V(x)}}\,\mu_V(dx) }{\mu_V(B(0,r_1))}\leqslant 1/2,$$ which gives us the
inequality \eqref{th3.2.2} with $C_1=2\left(c_3+\frac{c_6K(r_1)r_1^{2d+\alpha}}{k(r_1)}\right)$.

(2) Let $D$ be a bounded compact subset of $\R^d$. For any $f\in C_c^\infty(\R^d)$ such that supp$f\subset D$, we find
\begin{align*}\mathscr{D}_{\alpha,V}(f,f)=&\frac{1}{2}\iint_{\{D\times D,|x-y|>1\}} \frac{(f(x)-f(y))^2}{|x-y|^{d+\alpha}}\,dy\,\mu_V(dx)\\
&+\frac{1}{2}\iint_{\{D \times D^c, |x-y|>1\}}\frac{(f(x)-f(y))^2}{|x-y|^{d+\alpha}}\,dy\,\mu_V(dx)\\
&+\frac{1}{2}\iint_{\{D^c \times D, |x-y|>1\}}\frac{(f(x)-f(y))^2}{|x-y|^{d+\alpha}}\,dy\,\mu_V(dx)\\
\le&\iint_{\{D\times D,|x-y|>1\}}\frac{f^2(x)}{|x-y|^{d+\alpha}}\,dy\,\mu_V(dx)\\
&+\iint_{\{D\times D, |x-y|>1\}}\frac{f^2(y)}{|x-y|^{d+\alpha}}\,dy\,\mu_V(dx)\\
&+\frac{1}{2}\iint_{\{D\times D^c,|x-y|>1\}}\frac{f^2(x)}{|x-y|^{d+\alpha}}\,dy\,\mu_V(dx)\\
&+\frac{1}{2}\iint_{\{D^c\times D, |x-y|>1\}}\frac{f^2(y)}{|x-y|^{d+\alpha}}\,dy\,\mu_V(dx)\\
=&:\sum_{i=1}^4 J_i.
\end{align*}
Note that
\begin{equation*}
\begin{split}
 J_1 &\le  \int_{D} \bigg(\int_{\{|x-y|>1\}}\frac{1}{|x-y|^{d+\alpha}}\,dy\bigg)f^2(x)\,\mu_V(dx)\\
&\le \Big(\int_{\{|z|>1\}}\frac{1}{|z|^{d+\alpha}}\,dz\Big)\int_D
f^2(x)\,\mu_V(dx).
\end{split}
\end{equation*}
Since $$\int_{\{|x-y|>1\}}\frac{1}{|x-y|^{d+\alpha}}\,\mu_V(dx)\le
\int \mu_V(dx)=1,$$ we have
\begin{equation*}
\begin{split}
J_2 & \le \int_{D} \bigg(\int_{\{|x-y|>1\}}\frac{1}{|x-y|^{d+\alpha}}\,\mu_V(dx)\bigg)f^2(y)\,dy\\
&\le \Big(C_V^{-1}\sup_{y\in D}e^{V(y)}\Big)\int_D
f^2(y)\,\mu_V(dy).
\end{split}
\end{equation*}
Following the same way as above, we can get the similar estimates
for $J_3$ and $J_4$, respectively. Therefore, for every $f\in
C_c^{\infty}(\R^d)$ with supp$f\subset D$,
\begin{equation*}
\mathscr{D}_{\alpha,V}(f,f)\le C_{V,D}\mu_V(f^2),
\end{equation*}
where $C_{V,D}$ is a positive constant independent of $f$.

Thus, according to (\ref{th3.2.3}), for every $f\in
C_c^{\infty}(\R^d)$ with supp$f\subset D$,
\begin{equation*}
\begin{split}
&\mu_V(f^2) \le sC_{V,D}\mu_V(f^2)+\beta(s)\mu(|f|)^2.
\end{split}
\end{equation*}
By taking $s=\frac{1}{2C_{V,D}}$, we derive that
\begin{equation}\label{e17}
\mu_V(f^2)\le 2\beta\left(\frac{1}{2C_{V,D}}\right)\mu_V(|f|)^2.
\end{equation}

On the other hand, since the function $V$ is locally bounded, there
exist a point $x_0 \in D$ and a constant $r_0>0$ such that
$B(x_0,r_0)\subset D$, and $$\int_{B(x_0,r_0)}\mu_V(dx)\le
\left[{4\beta\left(\frac{1}{2C_{V,D}}\right)}\right]^{-1}.$$ Let
$f_0 \in C_c^{\infty}(\R^d)$ such that supp$f_0\subset B(x_0,r_0)$
and $f_0(x)>0$ for every $x \in B(x_0,{r_0}/{2})$. Hence, by the
Cauchy-Schwartz inequality,
\begin{equation*}
\mu_V(|f_0|)^2=\mu_V(|f_0|\I_{B(x_0,r_0)})^2\le
\mu_V(f_0^2)\mu_V(B(x_0,r_0)) \le
\frac{\mu_V(f_0^2)}{4\beta(\frac{1}{2C_{V,D}})}.
\end{equation*}
This along with (\ref{e17}) yields that
\begin{equation*}
\mu_V(f_0^2)\le \frac{1}{2}\mu_V(f_0^2).
\end{equation*}
However, due to the fact that $f_0(x)>0$ for $x\in
B(x_0,{r_0}/{2})$, $\mu_V(f_0^2)\neq 0$, which is a
contradiction, and so the super Poincar\'{e} inequality
(\ref{th3.2.3}) does not hold for $\mathscr{D}_{\alpha,V}$.
\end{proof}

\begin{remark} \label{remf} (1)
As the same way, we also can prove that the super Poincar\'{e}
inequality does not hold for the following Dirichlet form
\begin{equation*}
\mathscr{D}_{\rho,V}(f,f):=\frac{1}{2}\iint(f(x)-f(y))^2\rho(|x-y|)\,dy\, \mu_V(dx),
\end{equation*}
where $\rho$ is a positive measurable function on $\R_+$ such that
$\int_{(0,\infty)} \rho(r)r^{d-1}\,dr<\infty$ and $\sup\rho(r)<\infty$.

(2) As shown in Theorem \ref{th3.2} (1), if (\ref{th3.2.1}) holds,
then we can get the weighted Poincar\'{e} inequality for
$\mathscr{D}_{\alpha,V}$. However, different from the case for ${D}_{\alpha,V}$ (see \cite[Proposition 1.6]{CW}) and due to the lack of local super
Poincar\'{e} inequality for $\mathscr{D}_{\alpha,V}$, the global super Poincar\'{e} inequality
fails for $\mathscr{D}_{\alpha,V}$, which reveals that in some
situations, to derive the global super Poincar\'{e} inequality, the
local super Poincar\'{e} inequality is inevitable.
\end{remark}

\bigskip

\noindent{\bf Acknowledgements.} The authors would like to thank the referee and Professor Feng-Yu Wang for comments and suggestions. Financial support through the
project ``Probabilistic approach to finite and infinite dimensional
dynamical systems'' funded by the Portuguese Science Foundation
(FCT) (No.\ PTDC/MAT/104173/2008) (for Xin Chen), and National Natural Science Foundation of China (No.\ 11201073) and the Program for New Century Excellent Talents in Universities
of Fujian (No.\ JA11051 and JA12053)
(for Jian Wang) are gratefully acknowledged.

\end{document}